\documentclass[11pt]{article}
\usepackage[ansinew]{inputenc}
\usepackage{graphicx}
\usepackage{color}

\usepackage{enumerate,latexsym}
\usepackage{latexsym}
\usepackage{amsmath,amssymb}% CUIDADO: Si uso estos packages, $\widetilde{}$ no funciona a veces
\usepackage{graphicx}
\newfont{\bb}{msbm10 at 11pt}
\newfont{\bbsmall}{msbm8 at 8pt}
\def\rth{\mathbb{R}^3}
\def\R{\mathbb{R}}

\def\N{\mathbb{N}}
\def\Z{\mathbb{Z}}

\def\Hip{\mathbb{H}}

\def\D{\mathbb{D}}

\def\esf{\mathbb{S}}
\def\Te{\mathbb{T}}

\newcommand{\T}{\mathbb{T}}

\newcommand{\ben}{\begin{enumerate}}
\newcommand{\bit}{\begin{itemize}}
\newcommand{\een}{\end{enumerate}}
\newcommand{\eit}{\end{itemize}}
\newcommand{\wh}{\widehat}
\newcommand{\Int}{\mbox{\rm Int}}

\newcommand{\wt}{\widetilde}
\renewcommand{\S}{\Sigma}

\newcommand{\ed}{\end{document}}

\def\a{{\alpha}}

\def\t{{\theta}}
\def\g{{\gamma}}
\def\G{{\Gamma}}
\def\l{{\lambda}}

\def\de{{\delta}}
\def\be{{\beta}}
\def\ve{{\varepsilon}}
\def\cF{{\cal F}}

\def\cC{{\cal C}}

\def\cR{{\cal R}}

\def\cL{{\cal L}}
\newcommand{\ov}{\overline}

\def\centerbmp#1#2#3{\vskip#2\relax\centerline{\hbox to#1{\special
    {bmp:#3 x=#1, y=#2}\hfil}}}

\newtheorem{theorem}{Theorem}[section]
\newtheorem{lemma}[theorem]{Lemma}

\newtheorem{remark}[theorem]{Remark}
\newtheorem{corollary}[theorem]{Corollary}
\newtheorem{definition}[theorem]{Definition}

\newtheorem{assertion}[theorem]{Assertion}

\newenvironment{proof}{\smallskip\noindent{\it Proof.}\hskip \labelsep}
{\hfill\penalty10000\raisebox{-.09em}{$\Box$}\par\medskip}

\begin{document}
\begin{title}
{CMC foliations of closed manifolds}
\end{title}
%\vskip .2in

\begin{author}
{William H. Meeks III\thanks{This material is based upon
   work for the NSF under Award No. DMS - 1309236.
   Any opinions, findings, and conclusions or recommendations
   expressed in this publication are those of the authors and do not
   necessarily reflect the views of the NSF.}
   \and Joaqu\'\i n P\' erez
\thanks{The second  author was supported in part
by MEC/FEDER grant no. MTM2011-22547, and
Regional J. Andaluc\'\i a grant no. P06-FQM-01642.}
}
\end{author} \maketitle

\begin{abstract}
We prove that every closed, smooth   $n$-manifold $X$ %($n\geq 3$)
admits  a Riemannian metric together with a
CMC foliation if and only if its Euler characteristic is zero,
where by a CMC foliation we mean a  smooth, codimension-one, transversely
oriented foliation with
leaves  of constant mean curvature and where  the value of the constant
mean curvature can vary from leaf to leaf. Furthermore, we prove
that this CMC foliation of $X$ can be chosen so that when $n\geq 2$,
the constant values of the mean curvatures of its leaves change sign.
We also prove a general structure theorem for any such
non-minimal  CMC foliation of  $X$  that describes relationships between
the geometry and topology of the leaves, including the property that
there exist compact leaves for every attained value of the mean curvature.

\vspace{.3cm}


\end{abstract}

\section{Introduction.}
This manuscript studies  the existence, geometry and topology of smooth, transversely oriented
foliations $\cF$ of a smooth closed   Riemannian $n$-manifold $X$ (not necessarily orientable), such that
all of the leaves of $\cF$ are two-sided hypersurfaces of constant mean curvature and where the value of
the constant mean curvature can vary from leaf to leaf; here, that the foliation
is transversely oriented just means that there exists a smooth,  unit vector field $N$ on $X$
that is normal to the leaves of the foliation, and the convention of signs for the mean curvature $H$
of a leaf $L$ of $\cF$ is
\begin{equation}
  \label{eq:H}
  (n-1)H=\sum _{i=1}^{n-1}\langle \nabla _{E_i}E_i,N\rangle ,
\end{equation}
where $\langle ,\rangle $ is the ambient metric on $X$, $\nabla $
is the Riemannian connection for the induced
metric on $L$, and $\{ E_1,\ldots ,E_{n-1}\} $ is a local orthonormal frame for the tangent bundle of $L$.
Such a foliation $\cF $ is called a
{\em CMC foliation} of $X$; this is a particular case of a {\it tense} foliation, defined as a
smooth foliation of $X$ by $k$-dimensional submanifolds with parallel mean
curvature vector, see e.g., Definition 1.36 in
Rovenskii~\cite{rov1}. All manifolds and foliations appearing here will be assumed to be smooth
(of class $C^{\infty }$) unless otherwise stated.

By the next  theorem, the vanishing of the Euler characteristic of a  closed
$n$-manifold is equivalent to  the existence of a CMC foliation of the manifold with respect
 to some Riemannian metric.  In the case the $n$-manifold is orientable, this theorem
was proved by Oshikiri~\cite{osh3}; we emphasize that although in our presentation we give an
alternative proof of the orientable case that uses a previous result by Oshikiri (Theorem~\ref{thm2.4} below),
this approach is in fact not necessary, as we give another proof of  Theorem~\ref{main}
which does not use Oshikiri's results and which also covers the non-orientable case.
Furthermore, when $n\geq 3$ the CMC foliations $\cF$ that we construct
on a given closed $n$-manifold $X$ with vanishing Euler characteristic
satisfy that there are a finite number of components of the complement of the sublamination of minimal leaves
in $\cF$ such that each of these
foliated components is diffeomorphic to the product of an open $(n-1)$-disk $D$
and a circle $\esf^1$, with isometry group containing $SO(n-1)\times \esf^1$;
furthermore, the universal cover $D\times \R$ of each such component together with its
lifted foliation and metric are equivalent to a rather explicit CMC
foliation $\cF_n$ on  $D\times \R$ with a product metric $g_n$, such that this structure is
invariant under the action of $SO(n-1)\times \R$ and depends only on the dimension $n$.

Note that by definition, a CMC foliation is necessarily smooth.

\begin{theorem}[Existence Theorem for CMC Foliations] \label{main}
A closed   $n$-manifold admits a CMC foliation
for some Riemannian metric %$\cF$
if and only if its Euler characteristic is zero. When $n\geq 2$,
the CMC foliation can be taken to be non-minimal.
\end{theorem}

Since closed (topological) three-manifolds admit smooth structures and
the Euler characteristic of any closed manifold
of odd dimension is zero, the %statement and proof of the
previous theorem has the following corollary.

\begin{corollary} \label{cor1.2}
Every closed topological three-manifold %$X$
admits a smooth structure together with
a  Riemannian metric and a  non-minimal CMC
foliation.
\end{corollary}

Our proof of Theorem~\ref{main} is motivated by two seminal works. The first one, due to Thurston
(Theorem 1(a) in~\cite{th3}), shows
that a necessary and sufficient condition for a smooth closed   $n$-manifold $X$ to admit a
smooth, codimension-one foliation $\cF$ is for its Euler characteristic to vanish;
for our applications, we will need to check that $\cF$ can be chosen to be  transversely oriented.
The second one is the result by Sullivan (Corollary~3 in~\cite{sul1}) that given such a pair
$(X,\cF)$ where $\cF$ is orientable (see Definition~\ref{def2.3}),
then $X$ admits a smooth Riemannian metric
$g_X$ for which $\cF$ is a minimal foliation (i.e., $\cF$ is {\it geometrically taut}) if and only if
for every compact leaf $L$ of $\cF$ there exists a closed transversal that intersects $L$
(i.e., $\cF$ is {\it homologically taut}); for our applications, we will need to prove that
in the implication `homologically taut $\Rightarrow $ geometrically taut',
Sullivan's hypothesis that the foliation $\cF$ be orientable can be removed.

Theorem~\ref{main} is also related to the question of prescribing a mean curvature function for a given
transversely oriented, codimension-one foliation $\cF$ of a closed $n$-manifold~$X$: Walczak~\cite{walc1} asked
the question of which smooth functions $f$ on $X$ can be written as mean curvature functions of the leaves of
$\cF$ with respect to some Riemannian metric on $X$. In this line, Oshikiri~\cite{osh1,osh2}
characterized the solutions to this problem when $X$ is orientable; in particular, under the hypothesis
that $X$ is orientable, he described
for which functions $f\in C^{\infty }(X)$ that are constant
along the leaves of $\cF$, there exists a Riemannian metric on $X$ that makes $\cF$ a CMC foliation
(see Theorem~\ref{thm2.4} below for a precise statement of Oshikiri's result in the general setting).

An application of the divergence theorem shows that a geometrically taut foliation of a
closed   three-manifold $X$ cannot have a Reeb component\footnote{Classical Reeb components
of foliations of closed three-manifolds are defined
at the beginning of Section~\ref{sec3}. In Definition~\ref{reeb} we will describe the notion of
{\it enlarged Reeb component} of a codimension-one foliation, which makes sense in all dimensions
and differs for the classical definition in the case $n=3$ as it does not contain a single compact leaf.},
since such a foliation can clearly not have a compact leaf that separates the manifold.
Therefore, Novikov's theorem~\cite{nov1}, which implies that any foliation of the three-sphere
admits a Reeb component, also implies that the three-sphere does not admit any geometrically taut foliations.
By work of Novikov~\cite{nov1} and Rosenberg~\cite{rose5}
(also see Corollary~9.1.9 in Candel and Conlon~\cite{caco2}),
every closed  orientable three-manifold $X$ admitting
a codimension-one, transversely oriented foliation without Reeb components is either $\esf^2\times \esf^1$ or
irreducible\footnote{A three-manifold $X$ is {\it irreducible} if every embedded two-sphere
in $X$ bounds a three-ball in $X$.}; in particular, such an $X$ is a prime three-manifold.
This fact together with the divergence theorem imply that a closed non-prime three-manifold $X$
does not admit any geometrically taut foliations
%
%As homologically
%taut foliations on a closed three-manifold are free of Reeb components, we deduce that every
%closed, non-prime three-manifold does not admit any
%homologically taut foliations, and so by Sullivan's theorem such manifolds also
%do not admit any geometrically taut foliations
(i.e., such a manifold does not  admit any foliations that are minimal for some Riemannian ambient metric);
however, every closed   three-manifold  admits a Riemannian metric
together with a CMC foliation by Corollary~\ref{cor1.2}.

We now explain
the organization of the paper. In Section~\ref{sec2} we cover some of the basic
definitions related to CMC foliations.
In Section~\ref{secnew} we will prove Theorem~\ref{main} in the
case $n\leq 2$, so we will assume $n\geq 3$ for the remainder of the paper.
In Section~\ref{sec3} we study the existence of
codimension-one,  $(SO(n-1)\times \R)$-invariant CMC
foliations $\cR_{n-1} $ of the Riemannian product of the
 real number line $\R$ with the closed unit $(n-1)$-disk $\ov{\D}(1)\subset \R^{n-1}$
 with respect to a certain $SO(n-1)$-invariant metric, whose leaves are of one of two types:
the leaves that intersect $\D(r_1)\times \R $ (here $\D(r_1)=\{ x\in \R ^{n-1} \ | \ \|x\| <r_1\} $ and
$0<r_1<1)$ are rotationally symmetric
hypersurfaces which are graphical over $\D (r_1)\times \{ 0\} $ and asymptotic to the vertical
$(n-1)$-cylinder $\esf ^{n-2}(r_1)\times \R $; the remaining leaves of $\cR_{n-1} $ are the
vertical cylinders $\esf ^{n-2}(r)\times \R $, $r\in [r_1,1]$. All leaves of $\cR_{n-1} $
in $\D^{n-1}(r_1)\times \R $
are vertical translates of a single such leaf (in particular, they all have the same constant mean
curvature, equal to the constant value of the mean curvature of $\esf ^{n-2}(r_1)\times \R $),
while the (constant) mean curvature values of the cylinders $\esf ^{n-2}(r)\times \R $, $r\in [r_1,1]$,
vary from leaf to leaf. These foliations $\cR_{n-1} $ give rise under the quotient action
of $\Z\subset \R$ to what we call {\it enlarged foliated Reeb components} $\cR_{n-1}/\Z$,
that are diffeomorphic to $\overline{\D}(1)\times \esf^1$.

Section~\ref{sec:main} will be devoted to proving Theorem~\ref{main} along the following lines.
The sufficient implication follows directly from the divergence theorem
and the Poincar\'e-Hopf index theorem. As for the necessary
implication, the results in~\cite{th3} imply that
a smooth, closed   $n$-manifold $X$ with Euler characteristic zero admits a smooth,
transversely oriented foliation $\cF'$ of codimension one.
After a simple modification of $\cF'$ by the classical technique of turbularization
(explained in Subsection~\ref{sec4.1}), $\cF'$
can be assumed to have at least one non-compact leaf.
In Section~\ref{ConstructionF} we will prove the existence of a finite collection
$\Delta=\{\g_1,\ldots,\g_k\}$ of pairwise
disjoint, compact embedded arcs in $X$ that are transverse to the leaves of $\cF'$
and such that every  compact leaf of the foliation intersects at least one of
these arcs; this existence result will follow from work of
Haefliger~\cite{hae2} on the compactness of the set of compact leaves
of any codimension-one foliation of~$X$.
We will then proceed to modify $\cF'$ using again turbularization by introducing pairs of
what we called in the previous paragraph
``enlarged Reeb components'', one pair of these enlarged Reeb components for each $\g_i\in \Delta$.
These modifications give rise to a new transversely oriented foliation $\cF$ and a
related function $f\in C^{\infty }(X)$ that is constant
along the leaves of $\cF$ and that when $X$ is orientable, satisfies  Oshikiri's condition~\cite{osh1,osh2}
(see also Theorem~\ref{thm2.4} below) for
there to exist a Riemannian metric on $X$ that makes $\cF$ a CMC foliation.
Since the function $f$ also changes sign on $X$,
Theorem~\ref{thm2.4} will produce the ambient metric %$g_X$
on $X$ such that the foliation $\cF$ satisfies the
properties stated in Theorem~\ref{main} when $X$ is orientable. This analysis of the
orientable case for $X$ will be done in Section~\ref{subsecXorient}.
In  Section~\ref{subsec4.3} we will
give a direct proof of the necessary
implication of Theorem~\ref{main} that
avoids the results of Oshikiri
and also works when the manifold $X$ is  non-orientable;
this direct proof depends on the rotationally invariant foliations
constructed in Section~\ref{sec3}, Theorem~2 in Moser~\cite{moser1},
as well as on a generalization
of Sullivan's theorem to the case of non-orientable codimension-one
foliations, given in Theorem~\ref{thm-sul} of Section~\ref{sec4.5}.

Finally, in Section~\ref{sec5} we will prove the Structure Theorem~\ref{main2}
given below on the geometry and topology of non-minimal CMC foliations of
a closed   $n$-manifold. Before stating this theorem, we fix some notation.
For a CMC foliation $\cF$ of a (connected) closed  Riemannian
$n$-manifold $X$:
\bit
\item $N_\cF$ denotes the unit normal vector field to $\cF$ whose direction coincides
with the given transverse orientation.
\item  $H_{\cF}\colon X\to \R$ stands for the
{\em mean curvature function} of $\cF$ with respect to $N_{\cF}$.
 \item $H_\cF (X)=[\min H_{\cF},\max H_{\cF}]$ is the image of $H_\cF$.
 \item $\cC_\cF$ denotes the union of the compact leaves in $\cF$, which is a compact
 subset of $X$ by the aforementioned result of Haefliger~\cite{hae2}.
 \eit

\begin{theorem}[Structure Theorem for CMC Foliations] \label{main2}
Let $(X,g)$ be
a closed connected Riemannian $n$-manifold
which admits a non-minimal CMC foliation $\cF$. Then:
\ben
\item  \label{it1}  $\int_X H_{\cF} \,dV=0$ and so,
$H_{\cF}$ changes sign (here $dV$ denotes the volume element with respect to $g$).

\item  \label{it2} For $H$ a regular value of $H_\cF$,  $H_{\cF}^{-1}(H)$ consists
of a finite number of compact leaves of $\cF$ contained in $\Int(\cC_\cF)$.

\item  \label{it3} $X-\cC_\cF$ consists of a countable number of open
components and the leaves in each such component $\Delta $ have the same mean curvature
as the finite positive number of compact leaves in $\partial \Delta $;
furthermore, every leaf in the closure of $X-\cC_\cF$ is stable.
In particular, except for a countable subset of $H_\cF(X)$, every leaf of $\cF$
with mean curvature $H$ is compact, and for every $H\in H_\cF(X)$,
there exists at least one compact leaf of $\cF$ with mean curvature $H$.

\item \label{it4}
\ben[a.] \item Suppose that $L$ is a leaf of $\cF$ that contains
a regular point of $H_\cF$. Then $L$ is compact, it consists entirely of regular points of $H_\cF$ and lies
in the interior of $ \cC_\cF$. Furthermore, $L$ has % the nullity of 0 and the
index\footnote{This index (resp. nullity) is the number of
negative eigenvalues (resp. multiplicity of zero as an eigenvalue) of the
Jacobi operator of $L$ viewed as a compact, two-sided hypersurface with constant mean curvature in $X$.}
zero if and only if the function $g( \nabla H_\cF, N_\cF)=N_{\cF}(H_{\cF})$ is
negative along $L$, and if the index of  $L$
is zero, then it also has nullity zero.

\item  Suppose that $L$ is a leaf of $\cF$ that is disjoint from the regular points
of $H_\cF$. Then the index of $L$ is zero, and if $L$ is a limit
leaf\,\footnote{See Definition~\ref{def-limset} for
the definition of a limit leaf of a CMC lamination.} of
the CMC lamination of $X$ consisting of the compact leaves of $\cF$,
then $L$ is compact with nullity one.
\een
\item \label{it5} Any leaf of $\cF$ with mean curvature equal to $\min H_\cF$ or $\max H_\cF$
is stable and such a leaf can  be chosen to be compact with nullity one.
\een
\end{theorem}

\begin{remark}
\label{rem1}
{\em
Let $(X,g)$ and $\cF$ be as in Theorem~\ref{main2}.
\begin{enumerate}[(i)]
\item Item~\ref{it5} in Theorem~\ref{main2} implies restrictions
 on the Ricci curvature of $(X,g)$; namely,
 Ric$(N_\cF)\leq -(n-1)\max H_\cF^2$ at some point in $X$, which can
 be obtained from evaluating the index form of the Jacobi operator of a compact (stable) leaf
 with mean curvature equal to either $-\min H_{\cF}$, or
 $\max H_{\cF}$, at the constant function one on the leaf.

\item In the case that $n=3$, $X$ is orientable and $\cF$ is not topologically a
product foliation  of $X=\esf^2\times \esf^1$ by spheres, then  item 1 in Theorem~2.13 of~\cite{mpr19} and
item~\ref{it5} in Theorem~\ref{main2}
imply that the  scalar curvature of $(X,g)$ cannot be everywhere greater than
$-\frac23 \max H^2_{\cF}$.

%\item If the dimension $n$ of $X$ is three, then by Theorem~1.3 in~\cite{mpr21},
%there exist estimates for the norm of the
%second fundamental form of $\cF$ that depend only on an upper bound of the
%absolute sectional curvature of $(X,g)$; in particular, the positive numbers
%$|\!\min H_{\cF}|$, $\max H_{\cF}$ are bounded from above
%by constants that only depend on an upper bound of the norm of the absolute sectional
%curvature of $(X,g)$. For $n=4,5$, upper
%and lower bounds on the mean curvature of the leaves of a CMC foliation of $(X,g)$ are given in
%Theorem~5.23 in~\cite{mpr19}.
\end{enumerate}
}
\end{remark}

{\sc Acknowledgements.} The second author would like to thank Prof. Jes\'us A. Alvarez L\'opez
for valuable conversations about Oshikiri's results.

\section{Preliminaries.} \label{sec2}
\begin{definition}
\label{deflamination}
{\rm
 A smooth codimension-one {\it
lamination} of a Riemannian $n$-manifold $X$ is the union of a
collection of pairwise disjoint, connected, injectively immersed
hypersurfaces, with a certain local product structure. More precisely, it
is a pair $({\mathcal L},{\mathcal A})$ satisfying:
\begin{enumerate}
\item ${\mathcal L}$ is a closed subset of $X$;
\item ${\mathcal A}=\{ \varphi _{\be }\colon \D \times (0,1)\to
U_{\be }\} _{\be }$ is a maximal\footnote{Here maximality refers to atlases satisfying properties 2 and 3.}
atlas of (smooth) coordinate charts of $X$ (here
$\D $ is the open unit disk in $\R^{n-1}$, $(0,1)$ is the open unit
interval in $\R$ and $U_{\be }$ is an open subset of $X$). Charts in ${\mathcal A}$ are called
{\it lamination charts.}
\item For each $\be $, there exists a closed subset $C_{\be }$ of
$(0,1)$ such that $\varphi _{\be }^{-1}(U_{\be }\cap {\mathcal L})=\D \times
C_{\be}$.
\end{enumerate}

We will simply denote laminations by ${\mathcal L}$, omitting the
lamination charts $\varphi _{\be }$ unless explicitly necessary.
A smooth lamination ${\mathcal L}$ is said to be a {\it foliation of $X$} if ${\mathcal L}=X$
(and the corresponding charts in ${\mathcal A}$ are called {\it foliation charts}).
Every lamination ${\mathcal L}$  decomposes into a
collection of disjoint, connected smooth hypersurfaces (locally given by $\varphi
_{\be }(\D \times \{ t\} )$, $t\in C_{\be }$, with the notation
above), called the {\it leaves} of ${\mathcal L}$.
Note that if $\Delta
\subset {\cal L}$ is any collection of leaves of ${\cal L}$, then
the closure of the union of these leaves has the structure of a
lamination within ${\cal L}$, which we will call a {\it
sublamination.}

A smooth codimension-one lamination ${\cal L}$ of $X$ is said to be a {\it CMC lamination}
if each of its leaves has constant mean curvature (possibly varying
from leaf to leaf).
Given $H\in \R $, an {\it $H$-lamination} of $X$
is a CMC lamination all whose leaves have the same mean curvature
$H$. If $H=0$, the $H$-lamination is called a {\it minimal
lamination.} Observe that a closed hypersurface is a particular case of lamination,
hence $H$-laminations are natural generalizations of (closed) $H$-hypersurfaces.
 }
\end{definition}

\begin{definition}
\label{def-limset}
 {\rm
%Let $M$ be a complete, embedded hypersurface in a Riemannian $n$-manifold $X$. A
%point $p\in X$ is a {\it limit point} of $M$ if there exists a
%sequence $\{p_k\}_k\subset M$ which diverges to infinity in $M$ with
%respect to the intrinsic Riemannian topology on $M$ but converges in
%$X$ to $p$ as $k\to \infty$. Let lim$(M)$ denote the set of all limit
%points of $M$ in $X$; we call this set the {\it limit set of $M$}.
%In particular, lim$(M)$ is a closed subset of $X$ and $\overline{M} -M
%\subset \lim (M)$, where $\overline{M}$ denotes the closure of~$M$.
%
%The  notion of limit point can be extended to the case of
%a codimension-one lamination ${\cal L}$ of $X$ as follows:
Given a smooth codimension-one lamination $\cL$ of a Riemannian $n$-manifold $X$,
a point $p\in \mathcal{L}$ is a {\it limit point} if there exists a coordinate chart
$\varphi _{\beta }\colon \D \times (0,1)\to U_{\be }$ as in Definition~\ref{deflamination}
such that $p\in U_{\be }$ and $\varphi _{\be }^{-1}(p)=(x,t)$ with
$t$ belonging to the accumulation set of $C_{\be }$.
It is easy to show that if $p$ is a limit point of
a codimension-one lamination ${\cal L}$ (resp. of a   $H$-lamination), then the leaf $L$
of ${\cal L}$ passing through $p$ consists entirely of limit points of ${\cal L}$;
in this case, $L$ is called a {\it limit leaf} of ${\cal L}$.}
\end{definition}

\begin{definition}
\label{def2.3}
{\rm
Let $\cF$ be a codimension-one foliation of a manifold $X$.
$\cF$ is called {\em transversely orientable}
 if there exists a continuous, nowhere zero vector field whose
 integral curves intersect transversely to the leaves of $\cF$. Once such a vector field has been
 chosen, we call $\cF$ a {\it transversely oriented} codimension-one foliation.
% When $X$ is orientable with a fixed orientation, then this property is equivalent
%to have a choice of an atlas for $\cF$ as in Definition~\ref{deflamination}
%(with $C_{\be }=(0,1)$) so that the change of
% coordinates in the atlas preserves the orientation of the
% leaves; when $X$ is orientable with a fixed orientation, then this property is equivalent to choosing
% a continuous, nowhere zero vector field whose integral curves intersect transversely to the leaves of $\cF$.
$\cF$ is called {\it orientable} if there exists a smooth $(n-1)$-form on $X$ whose restriction to the tangent
spaces of leaves of $\cF$ is never zero. These notions can be expressed
in terms of the tangent bundle $T\cF$ and the normal bundle $T^{\perp }\cF$
to $\cF$ (once we have chosen a Riemannian metric in $X$): $\cF$ is orientable
(resp.transversely orientable) when $T\cF$
(resp. $T^{\perp }\cF$) is orientable.
}
\end{definition}

\begin{definition}
{\em
Let $\cF$ be a transversely oriented, codimension-one foliation of an $n$-manifold $X$ and
let $\g$ be a smooth simple closed curve contained in a leaf $L$ of $\cF$. After picking a
Riemannian metric $g$ on $X$, we can define a  {\em normal fence}  above $\g$ as follows.
Consider exponential coordinates for $(X,g)$ along $\g $. The {\em right} normal fence is the set
\[
A=\{ \exp _{\g (t)}(sN_{\cF }(\g(t)))\ | \ t\in \esf^1, \, s\in [0,\ve ]\} ,
\]
where $N_{\cF}$ is the positive unit normal vector to $\cF $ and $\ve >0$ is sufficiently small
so that $A$ defines an embedded annulus in $X$. If the parameter $s$ moves in $[-\ve ,0]$,
then we call the annulus a {\em left} normal fence.
}
\end{definition}

We finish this section by stating Oshikiri's theorem~\cite{osh2}
 about which smooth functions on a compact oriented $n$-manifold
$X$ can be viewed as mean curvature functions of a given
transversely oriented, codimension-one foliation $\cF$
on $X$. To state this result, we first need some notation. A
domain $D\subset X$ is called {\it saturated} if it is a union of leaves of $\cF$. A compact, smooth
saturated domain $D\subset X$ is called a $(+)$-{\it foliated compact domain} ($(+)$-fcd, for short)
if the transverse orientation of $\cF$ points outward everywhere on $\partial D$,
and we call $D$ a $(-)$-{\it foliated compact domain} (or $(-)$-fcd) if the transverse
orientation of $\cF$ is inward pointing everywhere on
$\partial D$. A smooth function $f\colon X\to \R $
is called {\it admissible for $\cF$} if there exists a Riemannian metric on $X$ such that $f$
is the oppositely signed mean curvature function of the
leaves of $\cF$ with respect to the given transversal orientation.
In other words for each $x\in X$,
$-f(x)$ is the mean curvature $H$ (defined as in equation (\ref{eq:H}))
of the leaf $L_x$ of $\cF$ passing through $x$ with respect
to the unit normal vector field to $L_x$ whose
direction coincides with the given transverse orientation.

\begin{theorem}[\cite{osh2}]
\label{thm2.4}
Let  $\cF$ be a codimension-one, transversely
oriented foliation in a compact oriented $n$-manifold $X$,
such that $\cF$ contains at least one $(+)$-fcd.
Then, a smooth function $f\colon X\to \R $ is admissible for $\cF$ if and only if
every minimal\,\footnote{Here, {\it minimal} refers to the partial order given by inclusion.}
$(+)$-fcd contains a point where $f$ is positive, and every $(-)$-fcd contains a point
where $f$ is negative.
\end{theorem}
Note that if a foliation $\cF$ as in Theorem~\ref{thm2.4} does not contain
any $(+)$-fcd, then it also does not contain any $(-)$-fcd
and thus Corollary~6.3.4 in Candel and Conlon~\cite{caco1} implies that $\cF$ is homologically taut.

\section{Proof of Theorem~\ref{main} in the case $n\leq 2$.}
\label{secnew}
Consider the curve $\a=\{(t, 3 +\cos t )\mid t\in \R\}$ in the $(x_1,x_2)$-plane and let $C$ in $\rth$
be the surface obtained by revolving $\a$ around the $x_1$-axis. Let $\cF$ be the
foliation of $C$ by circles contained in planes orthogonal to the $x_1$-axis, whose leaves have constant geodesic curvature, see Fig.~\ref{revolution}.
\begin{figure}
\begin{center}
\includegraphics[height=6cm]{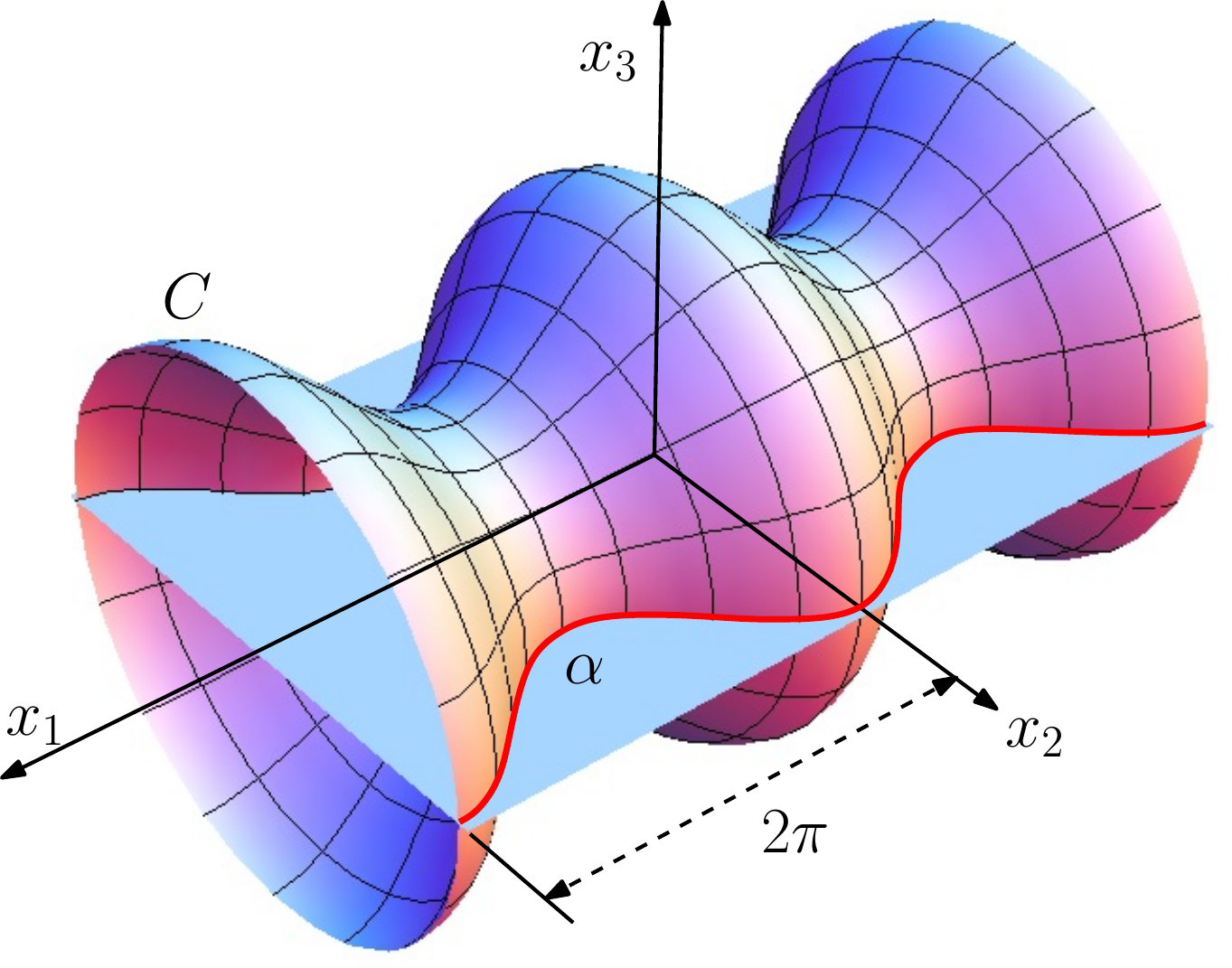}
\caption{The circles in $C$ foliate the surface by curves of constant geodesic curvature. The symmetry $R$ is
the composition of the reflection in the $(x_1,x_2)$-plane (depicted in the figure) with the translation by the vector $(2\pi ,0,0)$.}
\label{revolution}
\end{center}
\end{figure}
$\cF$ is transversely oriented by the normal vectors to the circles in $C$
that have positive inner product in $\rth$ with $\partial_{x_1}$.  Since the map  $R(x_1,x_2,x_3)= (2\pi +x_1,x_2,-x_3)$
preserves the orientation of the CMC foliation, then $\cF$ descends to a CMC foliation
of the Klein bottle $C/R$ or to the torus $C/(R^2)$.
By classification of closed surfaces, a closed surface with Euler characteristic zero
must be a torus or a Klein bottle.
Thus, Theorem~\ref{main} trivially holds when $n=2$.

In the remainder of this section we will produce  a Riemannian metric $g$ on a two-dimensional torus $M$
together with a  one-dimensional, transversely oriented
 foliation $\cF $ by curves  of constant geodesic curvature, where the geodesic
curvature function $H_{\cF}$ of the foliation changes sign and $\cF$ contains two Reeb components.
We will construct such a metric $g$ with a one-dimensional isometry group whose elements leave invariant $\cF$. The
construction will be made so that both $\cF$ and $g$ induce similar objects on a quotient Klein
bottle.

Consider a rectangle $R$ with sides $a_1,a_2,b_1,b_2$ identified as
usual to produce a flat torus $M$ after identification
of its sides; assume that the length of $a_1,a_2$ is 1. Consider vertical
lines $l_0,\ldots ,l_3$ as in Fig.~\ref{2dim}.
\begin{figure}
\begin{center}
\includegraphics[height=6.5cm]{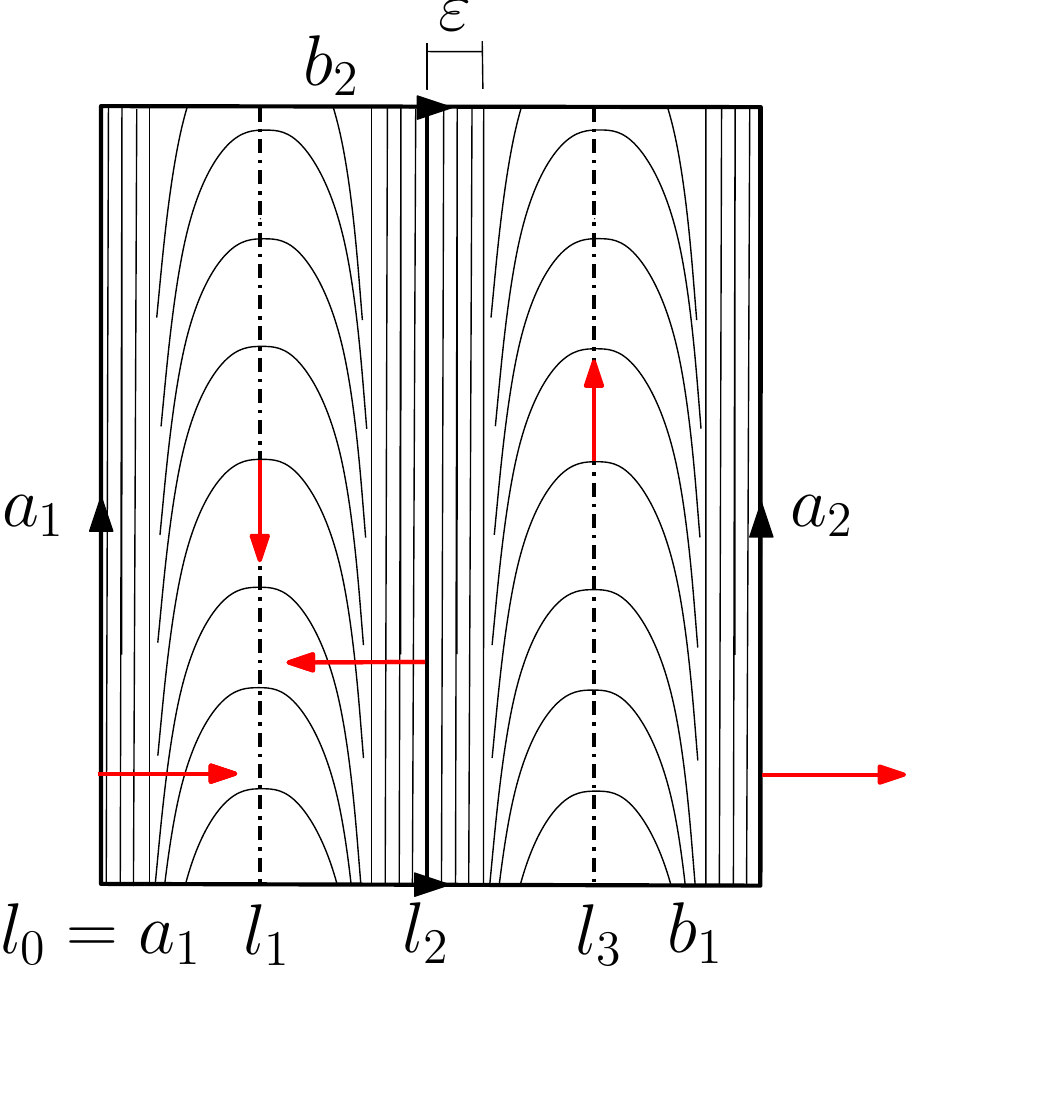}
\caption{The arrows indicate the transverse orientation on the foliation $\cF$.}
\label{2dim}
\end{center}
\end{figure}Let $G$ be the group of isometries of the
flat metric on $M$ generated by the reflections across $l_0,l_1$ and by the $\esf^1$-family of
 vertical translations; as a group, $G$
is isomorphic to the product of $\Z _2\times \Z_2\times \esf^1$.

We now describe the foliation $\cF $ of
$M$  depicted in Fig.~\ref{2dim}. Take $\ve >0$ sufficiently
small so that the closed vertical strips of width $\ve $ around the lines $l_0$ and $l_2$
are pairwise disjoint. $\cF $ restricts to the union $S$ of these strips as the product foliation
 by vertical lines; the restriction of $\cF$ to each of the open
strips in the complement of $S$ consists of one-dimensional Reeb foliations
invariant under $G$, see Fig.~\ref{2dim}. The transverse orientation we
choose for $\cF$ is described in the figure.

Consider a smooth function $f\colon M\to [-1,1]$ with the following properties.
\begin{enumerate}[(F1)]
\item $f$ is constant $-1$ (resp. $+1$) on the component of $M-S$ that contains $l_1$
(resp. $l_3$).
\item $f$ is invariant under the $\esf^1$-action by vertical
translations, and under the reflections in $l_1,l_3$.
\item If $\psi $ is the reflection in either $l_0$ or $l_2$, then $f\circ \psi =-f$.
\end{enumerate}

By Theorem~\ref{thm2.4},   $f$ is an admissible
function for $\cF$ in the sense explained at
the end of Section~\ref{sec2}.  However, we are interested in
obtaining the following more specific result. Given $i=0,\ldots ,3$, we
let $R_i$ denote the reflection in $M$ across $l_i$. Let $T_{\theta }$ be
the vertical translation by $\t \in [0,1)$ in $M$ (recall that the length of $a_1=1$).

\begin{theorem} \label{thm:2dim}
There exists a metric $g$ on $M$ that is invariant under the group $G$ such that
the function $-f$ is the geodesic curvature function of the foliation $\cF$. In particular,
for $\displaystyle I=T_{\frac12}\circ R_1$,
the quotient Klein bottle $M/I$ with its quotient metric
admits the transversely oriented CMC foliation $\cF/I$.
\end{theorem}
\begin{proof}
The key result used to prove
Theorem~\ref{thm2.4} stated above is Oshikiri's Main Theorem in~\cite{osh0}, and our proof
will also be based on this Main Theorem.
First note that $\int _Mf\, dV=0$, where $dV$ is the volume form on $M$
associated to the flat metric (with the usual orientation
induced from the one of $\R^2$), and so there exists a 2-form $\omega'$
such that $d\omega'=f\,dV$.
Since  $\int _Mf\, dV=0$ and the other homological condition in
item~3 of Oshikiri's Main Theorem in~\cite{osh0}  is satisfied for $\cF$, then
the application of the Hahn-Banach theorem in
the proof of  the Oshikiri's Main Theorem
guarantees the existence of a 1-form $\omega_1$ on $M$ such that
$d \omega_1=f\,dV$ and $\omega_1 $ restricted to any leaf of
$\cF$ is a nowhere zero one-form on the leaf.
Note that in order to verify that item~3 of Oshikiri's Main Theorem holds, it is easier  to
check that the following equivalent
condition  (3*) given on page 515 in ~\cite{osh1} holds:
\begin{enumerate}[(i)]
\item $\int_M f \, dV = 0$, and
\item  $\int_D f \, dV > 0$ for any $( + )$-fcd $D$.
\end{enumerate}

We next explain how to modify $\omega_1$ to obtain
another 1-form $\omega$ on $M$ that satisfies:
\begin{enumerate}
\item $T_\theta^*(\omega)=\omega$, for each $\t \in [0,1)$.
\item $R_i^*(\omega)=(-1)^{i+1}\omega$, $i=0,\ldots ,3$.
\item  $d \omega=f\,dV$ and $\omega$ restricted to any leaf of
$\cF$ is a nowhere zero one-form on the leaf.
\end{enumerate}

To obtain the desired 1-form, first average $\omega_1$ with respect
to the $\esf^1$-action given by the vertical translations
to obtain a 1-form $\omega_2$; in other words, $\omega_2(v_p)= \int_0^1 (T^*_{\theta}  \omega_1)(v_p)\, d\t $
for each tangent vector $v_p\in T_pM$; note
that $\omega_2$ also satisfies condition 3. Next successively define
\[
\omega_3=\frac12(\omega_2  +R_1^*\omega_2),
\ \omega=\frac12(\omega_3  -R_2^*\omega_3).
\]

We write the flat metric $g_0$ on $M$ as $g_0=g^1\oplus g^2$, where we are using the
$g_0$-orthogonal decomposition of $TM$ by the normal and tangent subbundles $T^{\perp}\cF$
and $T\cF$ with respect to $g_0$, where $g^1=g_0|_{T^{\perp }\cF}$ and $g^2=g_0|_{T\cF}$.
Then, we can express $\omega _p|_{T_p\cF}$ as a multiple of the length one-form $dl_{g^2}$ associated
to $g^2$, namely $\omega _p=\omega _p(JN_p)\, dl_{g_0}$, where $N_p$ is the unit normal
vector (with respect to $g_0$) given by the transverse orientation to $\cF$, and $J$ is the almost
complex structure on $M$ given by its orientation. Consider the
metric on $M$ given by
\[
g=\frac{1}{\omega (JN)^2}g^1\oplus \omega (JN)^2g^2
\]
with respect to the same decomposition of $TM$. It follows that
the volume form  associated to $g$ is $dV$, and the restriction of $\omega $
to $T\cF$ coincides with the length one-form of the restriction of $g$ to $T\cF$.
Under this definition of metric on $M$ and applying the so called Rummler's calculation
(see the beginning of Section~\ref{sec4.5} for more details), we
conclude that $-f$ is the geodesic curvature function of the metric $g$.
\end{proof}

\begin{remark}{\em
Theorem~\ref{thm:2dim} has an $n$-dimensional generalization that allows one to construct metrics on
$\esf^{n-1}\times \esf^1$, where the two vertical rectangles in Fig.~\ref{2dim} bounded by $l_0,l_2$ are each replaced
by the product of an $(n-1)$-dimensional 'horizontal' disk $D$, $n\geq3$, with a 'vertical' circle $\esf^1$; these
two domains $D\times \esf^1$ are foliated by enlarged Reeb components (see Definition~\ref{rem3.1} in the next section),
which are invariant under the action of vertical translations and reflections about vertical hyperplanes passing through
$l_1$ and $l_3$.
In the next section we will construct such CMC foliations so that the ambient metric is a product metric on
$D\times \esf^1$, for some metric on $D$; note that this type of metric in the case when $n=2$ would necessarily be flat,
which prevents the construction in Section~\ref{sec3} to work in dimension $n=2$.
}\end{remark}

\section{Rotationally symmetric $H$-foliations of Reeb type.}
 \label{sec3}
From this point on, we will always assume that $n\geq 3$.
\begin{definition}
\label{def3.1} {\em
Let $\D (R)$ the open disk of radius $R>0$ in $\R^{n-1}$. For $z\in \R$,
let $\Sigma _z\subset \R^{n-1}\times \R$ be the
graph of the function
\[
f +z\colon \D(1)\to \R,\quad
 f(x)=-\sec \left( \frac{\pi }{2}\| x\|^2\right) ,\quad x\in \D (1),
 \]
 and let $\S_\infty$ be the cylinder $\partial \D(1) \times \R\subset \R^n$.
\ben
\item  Consider the foliated closed cylinder $\ov{\D}(1)\times \R$ with leaves
$\S_z$, $z\in (-\infty, \infty]$. Up to diffeomorphism, any such foliation is called
a {\em Reeb-type foliation} of the cylinder.

\item After passing to the quotient by the natural action of
$\Z$ acting on $\ov{\D}(1)\times \R$ by  integer
translations in the $n$-th variable,
we obtain a compact $n$-manifold with boundary $(\ov{\D}(1)\times \R)/\Z $; we call this
foliated manifold a {\em Reeb-type foliation} of the solid torus $\ov{\D }\times \esf ^1$.

\item Given a foliation $\cF$ of
an $n$-manifold $X$, a compact saturated domain $R\subset X$ %$\cD\subset X$
is called a {\em Reeb component} of $\cF$ if it is diffeomorphically equivalent
to a Reeb-type foliation of $\ov{\D }\times \esf ^1$.\een}
\end{definition}

In this section we will construct a smooth Reeb type foliation enlarged by
a product foliation (see (E1), (E2) below) of a solid $n$-dimensional torus
together with a smooth,
rotationally symmetric ambient metric so that the leaves of the foliation
have constant mean curvature, and torus leaves sufficiently close to the boundary of the
solid cylinder are minimal; by torus leaves, we mean leaves that are diffeomorphic to $\esf^{n-2} \times \esf^1$.
The existence of some CMC foliation on the Reeb torus follows
easily from Oshikiri's Theorem~\ref{thm2.4}, although we will construct it explicitly by
imposing rotational symmetry, as the analysis of the corresponding ODE could be interesting
on its own right.

%%Assume for the moment  that $n=2$.
%\textcolor{blue}{{\bf Bill:} I have rewritten the remaining of this section as the $n$-dimensional
%case has non-trivial changes with respect to the ambient case $n=3$ that was studied in the
%last version: it is true that the symmetries of the problem still allow us to reduce the PDE to an
%ODE in the same orbit space, but the ODE itself depends on $n$ (see (\ref{eq:rot2})).
%Therefore, the first integral of the ODE and subsequent analysis also change. Please read
%carefully from here till the end of this section.}

We consider on $\ov{\D }=\ov{\D }(1)$ polar coordinates $r\in [0,1]$,
$\t \in \esf ^{n-2}(1)$. Given $H>0$, our goal is to construct an
{\it enlarged Reeb-type foliation} of the cylinder $\ov{\D }\times \R $; more precisely,
we will construct a smooth, $SO(n-1)$-invariant metric $ds^2$ over $\ov{\D }$
and a (smooth) $SO(n-1)$-invariant, CMC foliation $\cF$ of the Riemannian product
$(\ov{\D }\times \R,ds^2+dz^2)$, satisfying the following conditions:
 \begin{enumerate}[(E1)]
\item There exists $r_1\in (0,1)$ such that the leaves of $\cF$ in $\D (r_1)\times \R $
are of the form $\Sigma _z$, $z\in \R $, where $\Sigma _z$ is the
vertical translate by $z\in \R $ of a smooth, $SO(n-1)$-invariant $H$-hypersurface $\Sigma $
which is a graph over $\D (r_1)\times \{ 0\} $. Furthermore, the
generating profile curve $\G \subset \{ (r,z)\ | \ r\in [0,1),\, z\in \R \} $ of $\Sigma $
under the action of $SO(n-1)$ can be globally
parameterized as $\G (r)=(r,z(r))$, with $z\colon [0,r_1)\to \R $
being a smooth function that satisfies $z'(0)=0$, $z'(r)>0$ for all $r\in (0,r_1)$ and $z(r)\to
+\infty $ as $r\to r_1^-$. Note that the restriction
of $\cF$ to $\D(r_1)\times \R $ is an $H$-foliation
with respect to the ambient metric $ds^2+dz^2$.

\item The restriction of $\cF$ to $(\ov{\D }-\D (r_1))\times \R $ is the product foliation by
vertical $(n-1)$-dimensional cylinders $(\partial \D (R))\times \R  $, $R\in [r_1,1]$.
These cylinders also have constant mean curvature, which will vary from the value $H$ at $R=r_1$
to the value zero for all $R\in [r_2,1]$ (here $r_2\in (r_1,1)$).
Note that the restriction of $\cF$ to $\overline{\D}(r_1)\times \R $ is a Reeb-type foliation in the sense of
item~1 of Definition~\ref{def3.1}.
\end{enumerate}

\begin{definition}
\label{rem3.1}
{\rm
%The analysis to produce the $(SO(1)\times \R )$-invariant Reeb-type foliation $\cF$ of
%$\ov{\D }\times \R $ will be entirely performed in the orbit space $\{ (x,0,z)\ | \ x\in [0,1],\,
%z\in \R \} $ of $\ov{\D } \times \R $; in fact, the argument that follows can be easily
%generalized to the $n$-dimensional case, thereby producing
%an $(SO(n-1)\times \R)$-invariant foliation of $\D (r_1)\times \R$
% by vertical graphs $\Sigma _z=\Sigma +(\vec{0},z)$, $z\in \R$, so that
%each $\Sigma _z$ is an $SO(n-1)$-invariant $H$-hypersurface which is a graph over its
%vertical projection over $\D (r_1)$, together with the
%$(n-1)$-dimensional cylinders $(\partial \D (R))\times \R$, $R\in [r_1,1]$; this
%$(SO(n-1)\times \R )$-invariant foliation is what we will call an {\it enlarged Reeb-type foliation}
%of the cylinder $\ov{\D}\times \R$.
After taking quotients by the natural action of $\Z $ by translations in the vertical factor of $\ov{\D}\times \R $,
the quotient foliation of $\ov{\D }\times \esf^1$ will be called an {\it enlarged Reeb-type foliation}
of $\ov{\D }\times \esf^1$, see Fig.~\ref{enlarged}.
}
\end{definition}
\begin{figure}
\begin{center}
\includegraphics[height=6.8cm]{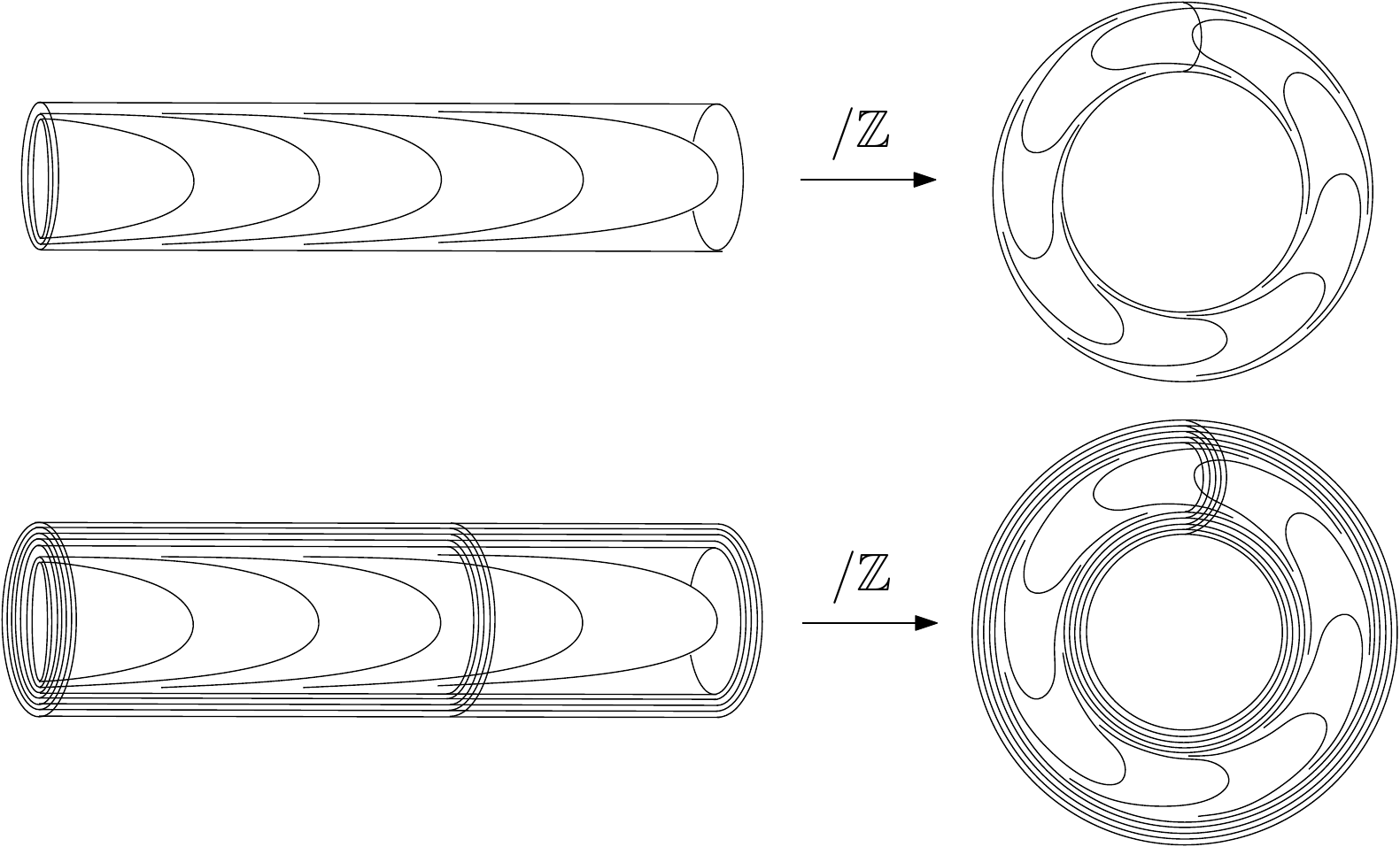}

\caption{Above: A Reeb-type foliation of a cylinder (left) and of $\ov{\D }\times \esf^1$ (right).
Below:  An enlarged Reeb-type foliation of a cylinder (left) and of $\ov{\D }\times \esf^1$ (right).}
\label{enlarged}
\end{center}
\end{figure}

%We return to the case $n=2$.
The following analysis of the ODE system that corresponds to a rotationally
 symmetric $H$-hypersurface in $\ov{\D }\times \R $ is inspired by the paper~\cite{FMP} by Figueroa,
 Mercuri and Pedrosa. Every smooth, rotationally symmetric
 Riemannian metric on $\ov{\D }$ can be written in the form
\begin{equation}
\label{rot0}
ds^2=dr^2+\phi (r)^2\, d\t ^2,
\end{equation}
where $d\t ^2$ denotes the metric of sectional curvature $1$ on $\esf^{n-2}(1)$ and
$\phi \colon [0,1]\to [0,\infty )$ is a smooth function that satisfies $\phi (r)>0$ if $r\in (0,1]$,
$\phi '(0)=1$ and all whose
derivatives of even order (including order zero) vanish at $r=0$. Choose local coordinates
$\t _1,\ldots ,\t _{n-2}$ on $\esf^{n-2}(1)$ so that $\{ \partial _{\t _i}
 =\frac{\partial }{\partial \t _i}\ | \ i=1,\ldots ,n-2\} $ is an orthonormal basis of
 $(T\esf ^{n-2}(1),d\t ^2)$.
Note that with respect to $ds^2$, $\partial _r=\frac{\partial }{\partial r}$ is unitary and orthogonal to
 each $\partial _{\t _i}$, and that $|\partial _{\t _i}|=\phi $.
The $(n-2)$-dimensional sphere $\esf^{n-2}(r)$ of radius $r\in (0,1]$ in $(\ov{\D },ds^2)$ is totally umbilic
with constant mean curvature
\begin{equation}
\label{eq:rot0}
\kappa _r=\frac{\phi '(r)}{\phi (r)}
\end{equation}
with respect to the inner pointing unit normal vector $-\partial _r$.

We consider on $\ov{\D }\times
\R $ the product metric $g=ds^2+dz^2$, where $z\in \R $ represents height in the
$\R $-factor. The group $SO(n-1)$ acts on $(\ov{\D }\times \R ,g)$ by isometries
(rotations about the $z$-axis $\{ r=0\} $), with orbit space
 ${\cal B}=\{ (r,z)\in \R^2\ | \ 0\leq r\leq 1\} $,
and the induced metric by $g$ on ${\cal B}^*=\{ (r,z)\ | \ 0<r\leq 1\} $ is flat. Suppose that
$s\mapsto \G (s)=(r(s),z(s))\in {\cal B}^*$ is a smooth curve parameterized by arc length.
Let us denote by $\sigma (s)$ the angle that the velocity vector $\dot{\G }(s)$ makes with
the  $\partial _r$ direction. The planar curvature of $\G $ is
$\kappa _{\G }(s)=\dot{\sigma }(s)$, and the Frenet dihedron of $\G $ can be written in coordinates with
respect to the orthonormal basis $\{ \partial _r,\partial _z\} $ as
\[
t=t(s)=(\dot{r},\dot{z})=(\cos \sigma ,\sin \sigma ),\quad
n=n(s)=(-\sin \sigma ,\cos \sigma ).
\]

Let $\Sigma $ be the smooth hypersurface of $\ov{\D }^*\times \R $ generated by $\G $ after
the action of $SO(n-1)$ about the $z$-axis, where  $\ov{\D }^*=\ov{\D }-\{ 0\} $. We want to relate the mean
curvature function ${\cal H}$  of $\Sigma $ with $\kappa _{\G }$ and $\kappa _r$.
As $\Sigma $ is $SO(n-1)$-invariant, we can do
this by computing the mean curvature of $\Sigma $ at a
point $p\in \Sigma \cap \{ (r,\vec{0},z)\ | \ 0<r\leq 1\} $, where $\vec{0}$ is the
origin in $\R^{n-3}$. After a slight abuse of notation,
we will identify $\Sigma \cap \{ (r,\vec{0},z)\ | \ 0<r\leq 1\} $ with $\G $.
The tangent space $T_p\Sigma $ at such a point $p\in \Sigma $ admits the orthonormal
basis $\{ X_i=\frac{1}{\phi }\partial _{\t _i}\ | \ i=1,\ldots ,n-2\} \cup \{ \dot{\G }\}  $ and thus,
\begin{equation}
\label{eq:rot1}
(n-1){\cal H }=\sum _{i=1}^{n-2}g\left( \nabla _{X_i}X_i ,N\right)+g\left(  \nabla _{\dot{\G }}\dot{\G },N\right),
\end{equation}
where $\nabla $ is the Riemannian connection of $g$ and  $N$ stands for the unit normal
vector to $\Sigma $ that coincides with $n$ at $p$ (the mean curvature function ${\cal H}$
is computed with respect to $N$). As $X_i$ is a horizontal vector field on the Riemannian product
$(\ov{\D }^*\times \R ,g)$, then $\nabla _{X_i}X_i$ is also horizontal and given by $\nabla _{X_i}X_i=
(\nabla ^{\ov{\D }}_{X_i}X_i,0)$, where $\nabla ^{\ov{\D }}$ is the Riemannian connection in
$(\ov{\D },ds^2)$. Since $N=n=-\sin \sigma \, \partial _r+\cos \sigma \, \partial _{z }$, then
\[
\sum _{i=1}^{n-2}g\left( \nabla _{X_i}X_i ,N\right)=-\sin \sigma \sum _{i=1}^{n-2}
ds^2(\nabla ^{\ov{\D }}_{X_i}X_i,\partial _r)=(n-2)\kappa _{r}
\sin \sigma \stackrel{(\ref{eq:rot0})}{=}(n-2) \frac{\phi '(r)}{\phi (r)} \sin \sigma ,
\]
where $r=r(p)$. As for the last term of (\ref{eq:rot1}), since the vertical plane $\{ y=0\} $
is totally geodesic in $(\ov{\D }\times \R ,g)$, then $g\left(  \nabla _{\dot{\G }}\dot{\G },N\right)
=\kappa _{\G }=\dot{\sigma }$. In summary, given $H\in \R $, the curve $\G =\G (s)$
generates after the action of $SO(n-1)$ an $H$-hypersurface of $(\ov{\D }^*\times \R ,g)$
if and only $r(s),z(s),\sigma (s)$ satisfy the ODE system
\begin{equation}
\label{eq:rot2}
\left\{ \begin{array}{rcl}
\dot{r}&=&\cos \sigma \\
\dot{z}&=&\sin \sigma \\
\dot{\sigma}&=&(n-1)H-(n-2)\frac{\phi '(r)}{\phi (r)} \sin \sigma .
\end{array}\right.
\end{equation}

The above system only makes sense in ${\cal B}^*$. Nevertheless, it is well-known
that solutions $\G $ of (\ref{eq:rot2}) that go to the inner boundary $\{ r=0\} $ of ${\cal B}^*$
must enter perpendicularly (see e.g., Eells and
Ratto~\cite{eer1} or Hsiang and Hsiang~\cite{hshs2}).
Three other basic properties of the system (\ref{eq:rot2}) are the following ones.
\begin{enumerate}[P1.]
\item Any vertical translation of a solution of (\ref{eq:rot2}) is also a solution
 of (\ref{eq:rot2}). In particular, (\ref{eq:rot2}) admits a first integral.
\item If a solution $\G (s)$ of (\ref{eq:rot2}) defined in $(s_0-\ve ,s_0]$ has
vertical tangent line at $s_0$, i.e., $\sigma (s_0)=\pm \pi /2$,
then $\G(s)$ can be extended by reflecting across
the horizontal line $\{ z=z(s_0)\} $ to a solution of
(\ref{eq:rot2}) defined in the interval $(s_0-\ve ,s_0+\ve ) $.
\item The vertical line $\G (s)=(r,s)$ (with $r\in (0,1]$ fixed) is a solution of (\ref{eq:rot2}),
and the value $H$ of the mean curvature of the corresponding vertical cylinder $C(r)=
(\partial \D (r))\times \R $  is $H=\frac{n-2}{n-1}\frac{\phi '(r)}{\phi (r)}$.
\end{enumerate}
Consider the smooth function $h\colon [0,1]\to \R $ given by
$h(r)=(n-1)H\int _0^r\phi (u)^{n-2}du$. It is straightforward to show that
\begin{equation}
\label{eq:rot3}
J(s)=\phi (r(s))^{n-2}\sin  \sigma (s)-h(r(s))
\end{equation}
is a first integral of (\ref{eq:rot2}), i.e., $J(s)$
is constant along any solution of (\ref{eq:rot2}).

\begin{remark}
{\rm
It is worth mentioning the following
geometric interpretation of the first integral $J$. The Killing vector field $\partial _z
=\frac{\partial }{\partial _z}$ on $(\overline{\D }\times \R ,g)$ induces a notion
of {\it scalar flux} on every $H$-hypersurface $\Sigma $ (not necessarily of revolution), defined as
\[
\mbox{Flux}(\Sigma ,\g )=\int _{\g }\langle \partial _z,\eta \rangle +(n-1)H\int _D
\langle \partial _z,N_D\rangle ,
\]
where $\g \subset \Sigma $ is a $(n-2)$-cycle on $\Sigma $, $D\subset \overline{\D }\times \R $
is a compact hypersurface with $\partial D=\g =D\cap \Sigma $ and $N_D$ is a unit normal vector field to $D$.
The divergence theorem shows that $\mbox{Flux}(\Sigma ,\g )$ only depends
on the homology class of $\g $ in $\Sigma $. In the particular case that $\Sigma $ is a hypersurface of
revolution and $\g $ is the $SO(n-1)$-orbit of a point in $\Sigma $, then
Flux$(\Sigma ,\g )=J$. This equality explains why $J=0$ is the value of the first integral
for CMC spheres (since $\g $ is null homologous in this case)
in the particular cases $n=3$,
$\phi (r)=r$ (which produces the
flat standard metric on $\ov{\D }^2$), $\phi (r)=\sin (r)$ (so $(\ov{\D},ds^2)$ is
isometric to a closed hemisphere in $\esf^2(1)$) and $\phi (r)=\sinh (r)$ (and $(\ov{\D},ds^2)$
can be viewed inside $\Hip ^2(-1)$).
}
\end{remark}
Rather than classifying all solutions of (\ref{eq:rot2})
by analyzing every possible value of $J$,
we will only study the case $J=0$ as we are interested in producing an example of an
$H$-hypersurface of $\ov{\D }\times \R $ that resembles in a certain sense, a graphical
CMC halfsphere; by this we mean the following. We will choose $\phi $ so that
the behavior of the profile curve $\G $ solution of (\ref{eq:rot2})
is similar to the one of a lower half-sphere, but with infinite length
and asymptotic to a vertical line instead of achieving the
vertical direction for its tangent line in a compact portion of curve.

Suppose that $\G (s) $ is a solution of (\ref{eq:rot2}) with $J=0$. Thus,
$\sin \sigma =h/\phi ^{n-2}$ and $\cos \sigma =\sqrt{1-\left( h/\phi ^{n-2}\right) ^2}$.
These equations only make sense if $h/\phi ^{n-2}$ takes values in $(-1,1)$.
With this observation in mind, it is easy to deduce from the equation $h'=(n-1)H\phi ^{n-2}$ and
L'H\^{o}pital's rule that
\begin{equation}
\label{eq:rot3a}
\lim _{r\to 0^+}\frac{h(r)}{\phi (r)^{n-2}}=0,\qquad
\lim _{r\to 0^+}\left( \frac{h}{\phi ^{n-2}}\right) '(r)=H.
\end{equation}
%$h(0)=0$, $h'(0)=2H\, f(0)=0$,
% $h''(0)=2H\, f'(0)=2H$. Hence,  L'H\^{o}pital's rule insures that
%$(h/f)(0)=0$, $(h/f)'(0)=H$.
This implies that the number
\[
r_1=\sup \left\{ r\in (0,1]\ : \ \frac{h(r)}{\phi (r)^{n-2}}\in (0,1) \mbox{ for all }r\in (0,r_1)\right\}
\]
exists (it clearly depends on the choice of $\phi $).

Note that if $r_1<1$, then $\frac{h(r_1)}{\phi (r_1)^{n-2}}=1$;
otherwise $r_1=1$ and we only can ensure that $\lim _{r\to 1^-}\frac{h(r)}{\phi (r)^{n-2}}\leq 1$.
Therefore, in $(0,r_1)$ we have
\begin{equation}
\label{eq:rot4}
\frac{dz}{dr}=\frac{dz}{ds}\frac{ds}{dr}=\frac{\sin \sigma }{\cos \sigma }=\frac{h/\phi ^{n-2}}{\sqrt{1-
\left( h/\phi ^{n-2}\right) ^2}}.
\end{equation}
From (\ref{eq:rot3a}) and (\ref{eq:rot4})  we deduce that
$z=z(r)$ is an increasing function of $r\in [0,r_1)$ (strictly increasing if $r>0$), whose graph
$\G $ intersects the vertical line $\{ r=0\} $ orthogonally.

We next compute the length of $\G $:
\begin{equation}
\label{eq:rot5}
 \mbox{length}(\G )_0^R=\int _0^R\sqrt{1+\left( \frac{dz}{dr}\right) ^2}dr\stackrel{(\ref{eq:rot4})}{=}
 \int _0^R\frac{dr}{\sqrt{1-\left( h/\phi ^{n-2}\right) ^2}},
\end{equation}
and we want to choose $\phi $ so that
${\displaystyle \lim _{R\to r_1^-}\mbox{length}(\G )_0^R=+\infty }$.
Note that if we impose
\begin{equation}
\label{eq:rot5a}
\frac{h}{\phi ^{n-2}}=\sqrt{1-(r_1-r)^2} \qquad \mbox{in }[r_1',r_1),
\end{equation}
for some $r_1'\in (0,r_1)$, then
\begin{equation}
\label{eq:rot5b}
\mbox{length}(\G )_{r_1'}^R\stackrel{(\ref{eq:rot5})}{=}\int _{r_1'}^R\frac{dr}{r_1-r}
=\log \left( \frac{r_1-r_1'}{r_1-R}\right)  ,
\end{equation}
which limits to $+\infty $ if $R\to r_1^-$, as desired. Equations (\ref{eq:rot5a}) and $h'=(n-1)H\phi ^{n-2}$
lead to an ODE of order one for $h$ whose solution is $h(r)=e^{(n-1)H \arcsin (r-r_1)}$. This function $h$ produces
\begin{equation}
\label{eq:rot6}
\phi \colon [0,r_1)\to \R , \quad \phi (r)=\frac{e^{\frac{n-1}{n-2}H\, \arcsin (r-r_1)}}{\left[ 1-(r_1-r)^2\right] ^{\frac{1}{2(n-2)}}}.
\end{equation}
We therefore conclude that with the choice of $\phi $ given by (\ref{eq:rot6})
for $r_1\in (0,1)$ fixed but arbitrary,
there exists a solution $\G  $ of (\ref{eq:rot2}) with first integral $J=0$,
that can be parameterized by $\G(r)=(r,z(r))$ with $z\colon (0,r_1)\to \R $ a smooth function that satisfies
$z'(r)>0$ for all $r\in (0,1)$ (by (\ref{eq:rot4}) and (\ref{eq:rot5a}))
and $z(r)\to +\infty $ as $r\to r_1^-$ (by (\ref{eq:rot5b})).

Note that with definition (\ref{eq:rot6}), $\phi $ fails to satisfy the conditions to define
a smooth metric at $r=0$ (for instance, $\phi (0)>0$). This problem can be overcome by truncating $\phi $
appropriately; to do this, fix $r_0\in (0,r_1)$ and substitute the definition of $\phi $ in (\ref{eq:rot6})
in $[0,r_0]$ by a smooth positive function so that the resulting function, also called $\phi $, is of class
$C^{\infty }$ at $r=r_0$ and $\phi '(0)=1$, $\phi ^{(2k)}(0)=0$ for all $k\in \N \cup \{ 0\} $, and
$h(r)/\phi (r)^{n-2}$ keeps taking values in $[0,1)$ for all $r\in [0,r_1)$. Plugging this new function
$\phi \colon [0,r_1)\to (0,\infty )$ in (\ref{rot0}) we define a smooth, $SO(n-1)$-invariant
metric in $\D (r_1)$ (not analytic). By uniqueness of solutions
of the ODE system (\ref{eq:rot2}) with given initial
values, we also conclude that restriction $\G |_{[r_0,r_1)}$ of the
graphical solution $\G $ appearing in the last paragraph
can be smoothly extended to $[0,r_1)$ as a solution of (\ref{eq:rot2}) with the
new function $\phi $. The resulting solution, relabeled as $\G(r)=(r,z(r))$,
is again a vertical graph over $[0,r_1)$ with $z'(0)=0$, $z(r)$ strictly increasing in $(0,r_1)$
 and $z(r)\to +\infty $ as $r\to r_1^-$.

We now extend  $\phi $ to $[r_1,1]$. Note that (\ref{eq:rot6}) defines a smooth function at
$r=r_1$, so simply extend $\phi $ smoothly to $[0,1]$ by a positive function in $[r_1,1]$, also
denoted by $\phi $, all whose derivatives coincide with the corresponding ones of  (\ref{eq:rot6}) at $r_1$.
Plugging this function $\phi $ in~(\ref{rot0}) we obtain a smooth, $SO(n-1)$-invariant metric
$ds^2$ on $\ov{\D }$. Recall that for any choice of $\phi $, the vertical cylinder $C(r)$ of radius $r$
in $(\overline{\D }\times \R, g)$ has constant mean curvature $H=\frac{n-2}{n-1}\frac{\phi '(r)}{\phi (r)}$.
As the graphical solution $\G (r)=(r,z(r))$, $r\in [0,r_1)$, constructed in the last paragraph generates an
$SO(n-1)$-invariant $H$-hypersurface $\Sigma $ which is smoothly asymptotic to the upper end of the vertical
cylinder $C(r_1)$, then the mean curvature of $C(r_1)$ is also $H$, so this must be the value of
$\frac{n-2}{n-1}\frac{\phi '(r)}{\phi (r)}$ (this can be checked by direct computation in (\ref{eq:rot6})).
Finally, we define the desired CMC foliation $\cF$ as the vertical translates of $\Sigma $ in $\D (r_1)\times \R $
together with the collection of vertical cylinders  $C(r)$, $r\in [r_1,1]$.

\begin{remark}
\label{rem3.3}
{\em
\ben[(A)]
\item Let $r_2\in (r_1,1)$. If we choose the extension of $\phi $ to $[r_1,1]$ so that it additionally
becomes constant in $[r_2,1]$, then the ambient metric on $(\ov{\D }(1)-\D(r_2))\times \R $ is flat and
the mean curvatures of the cylinders $C(r)$, $r\in [r_2,1]$ are zero.
Note that in the above construction, $H>0$, and $0<r_1<r_2$ are arbitrary.
\item The construction of the CMC foliation $\cF$ of $\ov{\D}(1)\times \R $ induces a CMC foliation of
$\ov{\D}(1)\times (\R/\l \Z )$ for any value of $\l >0$. Observe that the $(n-1)$-dimensional volume of the compact
hypersurface $(\partial \D (\de ))\times (\R /\l \Z )$ is independent of $\de \in [r_2,1]$ and of $H>0$,
and it can be prescribed arbitrarily by picking the appropriate value of $\l $.
\een
}
\end{remark}

\begin{definition}
\label{reeb}
{\em
We say that a codimension-one smooth foliation $\cF$ of a smooth $n$-manifold $X$ contains
an {\em enlarged Reeb component} $\Omega \subset X$  if $\Omega $ is a smooth saturated domain
and there exists a diffeomorphism $\phi \colon \ov{\D }\times \esf^1 \to \ov{\Omega }$
such that the pullback foliation $\phi ^*(\cF|_{\ov{\Omega }})$ is an
enlarged Reeb-type foliation of $\ov{\D }\times \esf^1 $, according to Definition~\ref{rem3.1}.
We refer to the image foliated region
$\phi (\ov{\D}\times \esf^1) \subset  X$ as an {\em enlarged Reeb component of $\cF$}.
Note that every enlarged Reeb component of $\cF $ contains a
classical Reeb component, according to Definition~\ref{def3.1}.
}
\end{definition}

\section{The proof of Theorem~\ref{main}.}
\label{sec:main}
Let $X$ be a closed   smooth $n$-manifold.
The existence of a  smooth, transversely oriented codimension-one foliation
of $X$ implies that the Euler characteristic of $X$ vanishes
(apply the Poincar\'e-Hopf index theorem
to the unit normal vector field to the transversely
oriented foliation with respect to an arbitrarily chosen metric on $X$).
Therefore the necessary implication in Theorem~\ref{main} is clear.
In fact, item~(a) of Theorem~1 in Thurston~\cite{th3} shows that $X$ admits a smooth foliation
if and only if it has vanishing Euler characteristic.
Since such an $X$ with Euler characteristic zero always admits a nowhere zero smooth vector field $V$
(Hopf~\cite{hf2}, also see Theorem 39.7 in Steenrod~\cite{sten1}),
%(this is clear when $n=2$,
%true for $n=3$ by Wood~\cite{wood1} and true for $n\geq4$ by
%Asimov~{\bf fix}),
then item (b) of Theorem~1 in~\cite{th3}
assures that the transversely oriented $(n-1)$-plane field on $X$ given
by the orthogonal complement of $V$ with respect
to a previously chosen metric on $X$, is homotopic to the tangent plane of a smooth, codimension-one
foliation. In particular, there exists a smooth, codimension-one, transversely oriented foliation on $X$.

Henceforth, to prove the sufficient implication of Theorem~\ref{main}
we can assume that $X$ admits a smooth, codimension-one,
transversely oriented foliation $\cF$.
We will also fix an auxiliary Riemannian metric $g$ on $X$. Our goal will be
to make a possibly different choice of $\cF$ and $g$
so that $\cF$ is a non-minimal CMC foliation with respect to $g$.

\subsection{Turbularization along an embedded closed transversal.}
\label{sec4.1}
Consider $\esf^1$ to be the quotient $\R/\Z$ with the orientation induced by the usual orientation on
$\R$ and let $\G\colon \esf^1 \to X$ be an embedded smooth curve transverse to $\cF$, which exists by
the following elementary argument. Consider a maximal integral curve $\g $
of the unit normal field $N_{\cF}$ to $\cF $. If $\g$ is closed, then we are done.
If $\g $ never closes, then the compactness of $X$ implies that $\g$ enters twice
(actually, infinitely many times) inside some product coordinate chart
$U=\D \times (0,1)$ of the foliation. Thus, $(X-U)\cap \g$
contains a subarc $\widehat{\g }$, one of whose end points $\widehat{\g }(0)$
lies in $\D \times \{ 0\} $ and the other one $\widehat{\g }(1)$ lies in $\D \times \{ 1\} $.
By basically joining
$\widehat{\g }(0)$ with $\widehat{\g }(1)$ by a `straight line segment' in $U$
and then smoothing the resulting embedded closed curve, we construct the desired embedded
closed transversal $\G \colon \esf^1\to X$ to $\cF $. Finally, after replacing $\G$ by a small perturbation
of its two-sheeted cover, we will assume that a small regular neighborhood of $\G$ is orientable.
We can also suppose that the inner product with respect to
$g$ of the velocity vector field to $\G $ with $N_{\cF}$ is positive.

Given $\sigma\in \{+,-\}$, $\G$ and $\cF$, we next describe a method for modifying
$\cF $ in a small regular neighborhood of $\G $ giving rise to a new foliation $\cF(\G,\sigma)$ of $X$,
%that is well-defined up to a diffeomorphism of $X$ isotopic to the identity.
by means of a standard technique called {\it turbularization.} Later we will introduce some
metric aspects in the turbularization process that will be useful in our goal to
prove the necessary implication of Theorem~\ref{main}.

First choose $\ve>0$ sufficiently small so that the closed embedded $8\ve$-neighborhood $V(\G ,8\ve )$
of $\G$ is parameterized by a diffeomorphism
\begin{equation}
\label{eq:Phi}
\Phi \colon \overline{\D}(8\ve )\times \esf^1\to V(\G ,8\ve ),
\end{equation}
where $\overline{\D}(r)=
\{ x\in \R^{n-1}\ | \ \| x\| \leq r \}$ for each $r>0$.
We can also take $\Phi $ so that the restricted foliation
$\cF |_{V(\G ,8\ve )}$ of $V(\G ,8\ve )$ consists of
$\{ \Phi (\overline{\D}(8\ve )\times \{\theta \})\ | \
\theta \in \esf^1\} $; in particular, all of these leaves are
$(n-1)$-dimensional disks. Note  that each orbit of
the action of $\esf^1$ on  $V(\G ,8\ve )$,  induced by pushing forward
via $\Phi $ the product action of $\esf^1$ on $ \overline{\D}(8\ve )\times \esf^1$,
intersects each leaf of $\cF |_{V(\G ,8\ve )}$
transversely in a single point. In particular, given $r\in (0,8\ve ]$,
the compact hypersurface
\[
\T(r)=\Phi (\partial \ov{\D }(r)\times \esf^1)
\]
obtained as the orbit of the action of $\esf^1$ on
$\partial \D (r)$, intersects each of the disk-type
leaves of $\cF |_{V(\G ,8\ve )}$  transversely in an embedded $(n-2)$-sphere.
$SO(n-1) \times \esf^1$ acts naturally on
$\overline{\D }(8\ve )\times \esf ^1$ (hence on
$V(\G ,8\ve )$ via $\Phi $) so that each $(A,\theta )\in SO(n-1) \times \esf^1$
acts by rotation by $A$ in $\overline{\D }(\ve )$
and by translation by $\theta $ in $\esf^1$.

Given $\sigma\in \{+,-\}$, let $S_\sigma$ be a connected, complete, non-compact
smooth hypersurface with compact
boundary in $V(\G ,8\ve )$ with the following properties:
\begin{enumerate}[S1.]
\item $S_{\sigma }$ is contained in
$\Phi [(\ov{\D }(8\ve)-\overline{\D }(4\ve))\times \esf^1)$ and it is of
revolution, i.e., if $\Phi (x,\t)\in S_{\sigma }$, then $\Phi (Ax,\t)\in S_{\sigma }$
for any $A\in SO(n-1)$.

\item $S_{\sigma }$ is graphical
(with respect to the orbits of the $\esf^1$-action on the second factor)
over the annulus $\Phi [(\D (6\ve )-\overline{\D }(4\ve))\times \esf^1)$.

\item The intersection of
$S_{\sigma }$ with $\Phi [(\overline{\D }(8\ve)-\overline{\D }(6\ve )\times\esf^1]$
equals the annulus
$\Phi [(\overline{\D}(8\ve )-\overline{\D}(6\ve ))\times \{0\}]$ (recall that
this annulus is part of a leaf of $\cF |_{V(\G ,8\ve )}$).

\item $S_{\sigma }$ is smoothly asymptotic in $V(\G ,8\ve )$ to the hypersurface $\T(4\ve)$.

\item The orientation of $\T(4\ve)$ induced by the
annular graph $S_{\sigma }$ coincides with the inward pointing unit normal to the
solid region $\Phi (\D (4\ve )\times \esf^1)$ when $\sigma=+$, and with
the outward pointing unit normal when $\sigma=-$;
see Fig.~\ref{fig1} for the case of $S_{\sigma}=S_+$.
\end{enumerate}
\begin{figure}
\begin{center}
\includegraphics[height=5.8cm]{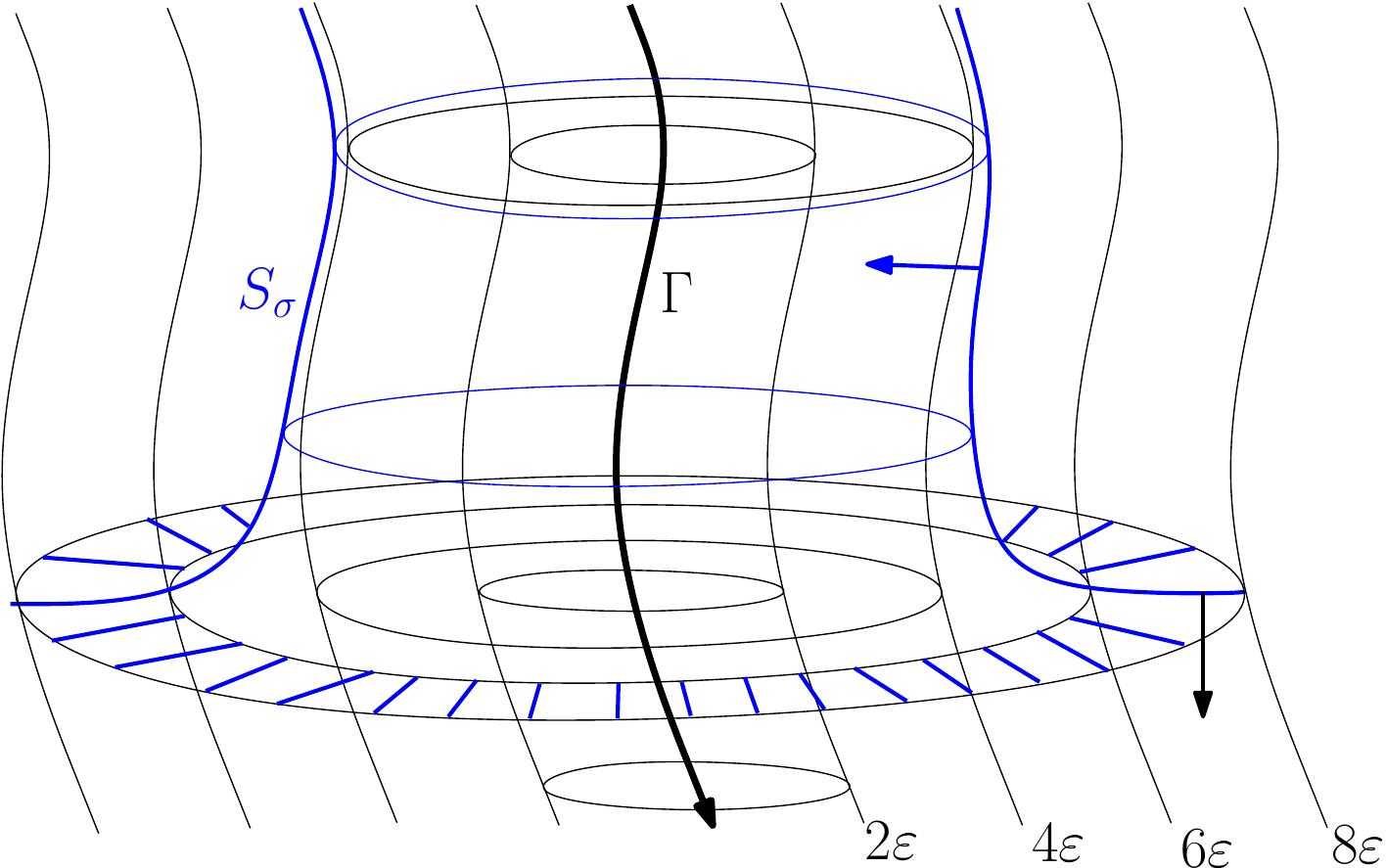} \label{fig1}
\caption{The complete non-compact hypersurface $S_{\sigma }=S_{+}$ in an
$8\ve$-regular neighborhood of a closed transversal~$\G $.
Note that as $\G $ is closed, then $S_{\sigma }$ comes
again into the represented portion of $X$, lying in between
the represented portion of itself and the compact hypersurface
$\Te (4\ve )$, and this happens infinitely many
times as $S_{\sigma }$ wraps infinitely often around $\T(4\ve)$
limiting to this compact hypersurface.}
\end{center}
\end{figure}

We are now in a position to describe $\cF(\G,\sigma)$.
Note that Property S2 above implies that
the family of translates $S_\sigma(\t)$ of $S_\sigma$ by
elements $\t\in\esf^1$ in the second factor
of the $(SO(n-1) \times \esf^1)$-action on $V(\G ,8\ve)$, are all disjoint.
These translates $S_{\sigma }(\t )$, $\t \in \esf^1$, together with the
hypersurfaces $\T(t)$, $t\in (2\ve ,4\ve]$
and with the leaves of $\cF\cap [X-\Phi (\D (6\ve)\times \esf^1)]$
produce a smooth, transversely oriented foliation of
$X- \Phi (\overline{\D }(2\ve)\times \esf^1)$.
In turn, this foliation can be extended to a smooth, transversely oriented
foliation $\cF(\G,\sigma)$ of $X$ by attaching an enlarged
Reeb component $\Omega =\Omega (\G ,\sigma )$ along the boundary $\T (2\ve)$ of
$X- \Phi (\overline{\D }(2\ve)\times \esf^1)$
and so that the $\esf^1$-action on
$\Phi ([\overline{\D }(8\ve)-\overline{\D }(2\ve)]\times \esf^1)$
agrees with the natural translational action of $\esf^1$ on $\Omega $.

\subsection{Constructing the desired foliation of $X$.}
\label{ConstructionF}
Using the above construction of $\cF(\G,\sigma)$ from $\cF$, $\G$ and $\sigma $, we will show
how to obtain a smooth, transversely oriented, codimension-one foliation of $X$ such that
with respect to some Riemannian metric on $X$ (to be defined in Sections~\ref{subsecXorient}
and \ref{subsec4.3}), has leaves of constant mean curvature.
We start with a smooth, transversely oriented, codimension-one foliation $\cF$ on $X$.
After possibly doing turbularization along a closed transversal to $\cF $ (such a closed
transversal was proven to exist at the beginning of Section~\ref{sec4.1}), we can assume that
 $\cF $ contains a non-compact leaf.
 %If every leaf of $\cF $ is compact, then
%$\cF$ is equivalent to a bundle over a circle
%(this is a consequence of the Reeb Stability Theorem, see
%e.g., Theorem~2.4.1 in~\cite{caco1});
%in this case, there exists a closed transversal
%$\G \subset X$ which intersects every leaf of $\cF$ in a single point.
%After doing turbularization
%along $\G $, we will henceforth assume that $\cF$ contains a non-compact leaf.

\begin{lemma}
\label{lemma4.1}
There exists a finite collection $\Delta =\{\g_1,\ldots,\g_k\}$ of pairwise
disjoint, compact embedded arcs that are transverse to the leaves of $\cF$
and positively oriented with respect to $N_{\cF }$, such that:
\begin{enumerate}
\item Every  compact leaf of $\cF$ intersects at least one of the $\g_i$.
\item For each $i=1,\ldots ,k$, the end points of $\g _i$ lie on non-compact leaves
of $\cF$.
\end{enumerate}
\end{lemma}
\begin{proof}
Let ${\cF }_0$ be the set of compact leaves of ${\cF }$. After picking a metric on $X$,
we can endow ${\cF }_0$ with the structure of a compact metric space with the induced
distance between compact leaves (compactness of ${\cF }_0$ follows from Haefliger~\cite{hae2},
also see Theorem~6.1.1 in~\cite{caco1}).  Given a leaf $L\in {\cF }_0$,
then either $L$ lies in a maximal, compact oriented 1-parameter family
$I=\{ L_t\ | \ t\in [0,1]\} \subset {\cF }_0$ or it fails to have this property.

In the first case,  the leaves $L_0,L_1$ are limits of non-compact leaves of ${\cF }$.
Pick points $p(0),p(1)$ in non-compact leaves of $\cF$ sufficiently
close to $L_0,L_1$ so that there exists a positively oriented transversal arc $\g _L$ joining
$p(0)$ with $p(1)$ and
which intersects exactly once each of the leaves in $I$. If on the contrary,
 $L$ does not lie in any compact oriented 1-parameter family $I$ as before, then a slight
 modification of the above arguments applies to $L$ in order to find an arbitrarily short,
 positively oriented transversal arc $\g _L$ intersecting $L$ exactly once and
with end points $p(0),p(1)$ in non-compact leaves of $\cF$.

Hence we have associated to
each $L\in {\cF }_0$ an open arc $\g _L$ that intersects $L$
with end points in $X-\cC _{\cF}$,
where $\cC _{\cF}=\cup _{L\in {\cF }_0}L$.
For each open transversal arc $\g _L$  as before, consider the set of leaves $A(\g _L)$ in
${\cF }_0$ that intersect $\g _L$. Clearly $A(\g _L)$ is an open set in
${\cF }_0$. As ${\cF }_0$ is compact, then we can extract a finite
open subcover from the family $\{ A(\g _L) \ | \ L\in {\cF }_0\} $, and the lemma follows
by choosing the transversal arcs associated to the finite subcover.
\end{proof}

Our next goal is to modify $\cF$ by turbularization along pairs $\G _1^i,\G_2^i$
of disjoint closed transversals associated to each $\g _i\in \Delta $
with the notation of Lemma~\ref{lemma4.1}. We next explain how to associate
such a pair $\G_1^i,\G_2^i$ to $\g _i$.
 Consider one of the oriented arcs $\g \in \Delta $ appearing in Lemma~\ref{lemma4.1}.
Note that $\g $ intersects $\cC _{\cF}$ in a compact set.  As we travel along $\g$ with its orientation,
 there exists a first point $p(\g )$ in $\g $ that lies in a compact leaf of $\cF$, and a last
 point $q(\g )$ in $\g $ that lies in a compact leaf of $\cF$. We label by $L(p(\g )),
 L(q(\g ))$ the (compact) leaves of $\cF$ passing through $p(\g )$ and $q(\g )$, respectively. Note that

\begin{itemize}
 \item $L(p(\g ))$ has non-trivial holonomy on its side  containing
 the starting point of $\g $.
  \item $L(q(\g ))$ has non-trivial holonomy on its side  containing
 the end point of $\g $.
 \end{itemize}

 A brief explanation of the above expression 'non-trivial holonomy' is in order. We will explain this
 property at $L=L(p(\g ))$; the case at $L(q(\g ))$ is similar. Suppose that
 $\varphi _{\be }\colon \D \times  (0,1)\to U_{\be }$ is a foliation chart
 (we use the notation in Definition~\ref{deflamination}), so that $\varphi _{\be }
 (\D \times \{ 1/2\} )\subset L$ and $\varphi _{\be }(\vec{0},1/2)=p(\g )$, where $\vec{0}$ denotes the
 origin in $\D $. The component $\g _1$ of $\g \cap U_{\be }$ that passes through $p(\g )$ can be
 identified with $\{ \vec{0}\} \times (0,1)$. Given an embedded, smooth transversal arc $\a \subset U_{\be }$
 with end points $\a (0)\in \D \times \{ 0\} $, $\a (1)\in \D \times \{ 1\} $, the local product
 structure of the foliation chart determines a diffeomorphism
 $ h\colon \g _1\to \a $ so that
 \begin{equation}
 \label{hol1}
 \varphi _{\be }^{-1}\left( h(\varphi _{\be }(\vec{0},t))\right) \in \D \times \{ t\} ,\quad \mbox{for all }
 t\in (0,1).
 \end{equation}

If $c\colon [0,1]\to L$ is a path starting at $p(\g )$, then we can cover $c([0,1])$ by
a finite number of foliation charts $(U_{\be _1},\varphi _{\be _1}), \ldots ,
(U_{\be _k},\varphi _{\be _k})$ as before, so that $U_{\be _i}\cap U_{\be _{i+1}}\neq
\mbox{\O }$ for all $i$. Pick embedded, smooth transversal
arcs
\[
\a _1=\g _1\subset U_1,\
\a _2\subset U_{\be _1}\cap U_{\be _2},\ldots ,
\a _{k-1}\subset U_{\be _{k-1}}\cap U_{\be _k},\
\a _k \subset U_{\be _k}
\]
 with  $\a _i(1/2)\in c([0,1])$ for all $i$, $\a _k(1/2)=c(1)$ and end points
$\a _i(0)\in \D \times \{ 0\} $, $\a _i(1)\in \D \times \{ 1\} $.  Then,
 (\ref{hol1}) produces
diffeomorphisms $h_i\colon \a _i\to \a _{i+1}$,
$1=1,\ldots ,k-1$. The composition $h_{k-1}\circ \ldots \circ h_1$ is a diffeomorphism
from the germ of $\a _1$ at $c(0)=p(\g )$ to the germ of $\a _k$ at $c(1)$, called
{\it holonomy transport.} It is a well-known fact that the holonomy transport only
depends on the relative homotopy of $c$. In the particular case that $c$ is a
loop based at $p(\g )$, the homotopy invariance of the holonomy transport
gives rise to a homomorphism
\[
\mbox{Hol}\colon \pi _1(L,p(\g ))\to G(\g )
\]
from the fundamental group of $L$ based at $p(\g )$ to the group of germs of
diffeomorphisms of the transversal arc $\g $ at $p(\g )$. The `non-trivial holonomy'
mentioned above means that the homomorphism Hol is not constant equal to the
identity element in $G(\g )$.

As $p(\g )$ is the first point in $\g$ that lies in a compact leaf of $\cF$, then
there exists a closed embedded curve $c_1\subset L(p(\g))$ whose homotopy
class $[c_1]\in \pi _1(L,p(\g ))$ produces a nontrivial diffeomorphism Hol$([c_1])
\in G(\g )$; more specifically, there exists a left normal fence
above $c_1$  and non-compact leaves of $\cF$ limiting to $L(p(\g ))$ on the local
side of $L(p(\g ))$ that contains the initial  point
of $\g $. In fact, after a small perturbation of $\g$ we
may assume that this normal fence is
chosen to contain the subarc $\g |_{p(\g )}^0$ of $\g $
between the beginning point of $\g $ and $p(\g )$
(after possibly shortening $\g $).
A standard argument allows us to find a positively oriented
closed transversal $\G _{p(\g )} $ to $\cF$
that lies in the left normal fence, intersects
$\g |_{p(\g )}^0$ at a single point and is topologically
parallel to $c_1$, see Fig.~\ref{closedtransversal}.
\begin{figure}
\begin{center}
\includegraphics[height=4.5cm]{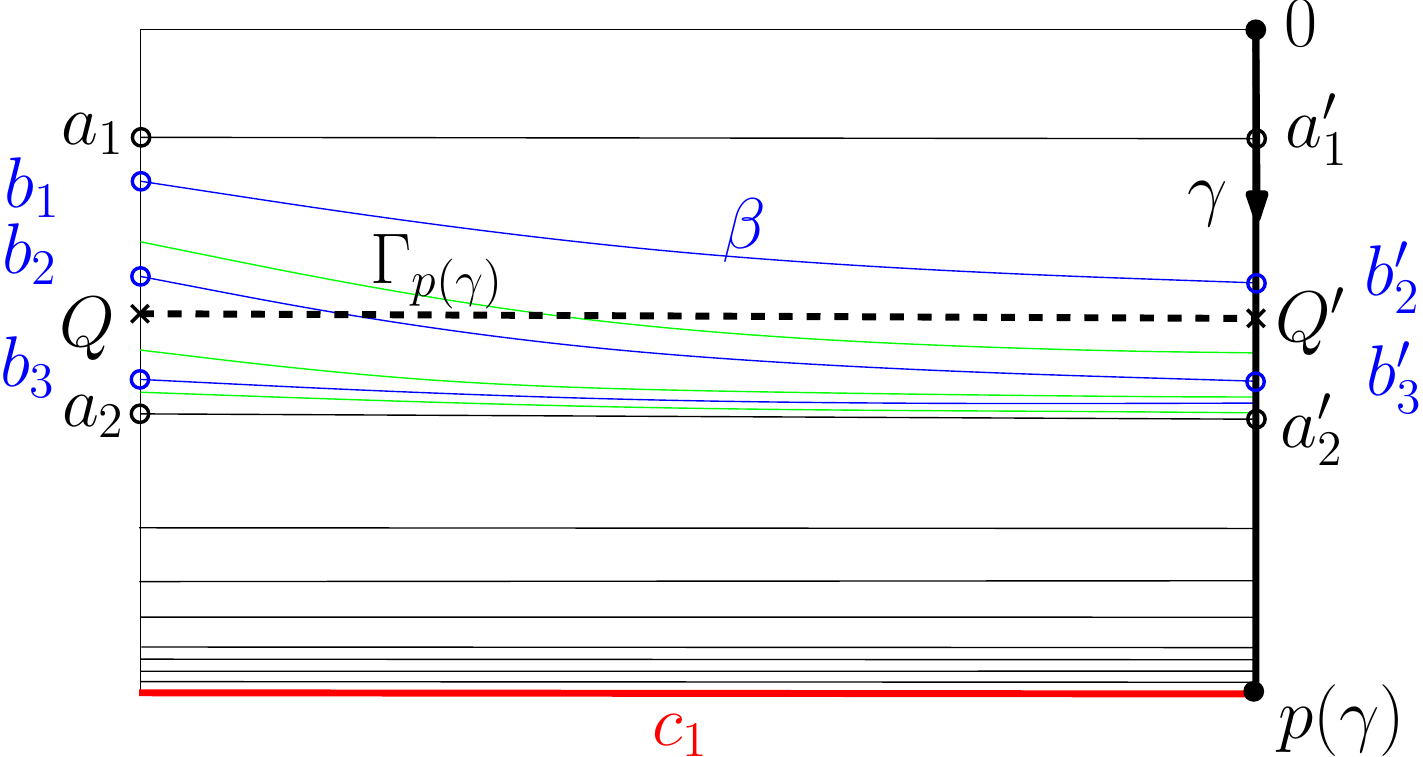}
\caption{A left normal annular fence $A$ above $c_1$ (identify
the two ``vertical sides'' of $A$ by horizontal translation)
transversely intersects infinitely many leaves
of $\cF$ different from $L(p(\g ))$, all of them non-compact. The non-trivial holonomy of
$c_1$ in $A$ implies that at least one of the intersection curves of $A$ with
leaves of $\cF $ is a non-compact
embedded spiraling curve $\beta $ (the points $b_j,b_j'\in \beta $
identify for each $j\in \N$) that limits in $A$ to some closed intersection
curve (with end points $a_2,a_2'$ in the picture) which could be $c_1$.
Spiraling intersection curves near $c_1$ are then grouped
into open annular components in $A$, each of which
is bounded by closed intersection curves; the closed transversal $\G _{p(\g )} $
can be then constructed by transversely crossing every
spiraling intersection curve in a given annular component;
note that the two end points of the above representation
of $\G _{p(\g )} $ identify to a single point $Q=Q'$.}
\label{closedtransversal}
\end{center}
\end{figure}

Let $\cF(\G _{p(\g )} ,-)$ be the smooth foliation of $X$ constructed by
turbularization around $\G_{p(\g )}$ as we explained in Section~\ref{sec4.1}; in particular,
$\cF (\G _{p(\g )},-)$ contains an enlarged Reeb component $\Omega (\G _{p(\g )},-)
\subset \Phi (\ov{\D }(3\ve )\times \esf ^1)$ (here we are using the notation in
(\ref{eq:Phi})) and $\cF(\G _{p(\g )},-)$ coincides with the previous foliation $\cF $
outside $\Phi (\ov{\D }(6\ve )\times \esf^1)$.

We next do a similar construction on the ``future'' side of $q(\g)$, i.e.,
 there exists a closed embedded curve $c_2\subset L(q(\g))$, a right normal fence
above $c_2$ containing the subarc $\g|_1^{q(\g )}$ of $\g $ between $q(\g )$
and the ending point of $\g$, and a positively oriented closed transversal
$\G _{q(\g )}$ to $\cF$ that lies in the normal fence, intersects
$\g|_1^{q(\g )}$ at a single point and is topologically parallel to $c_2$,
and then we do turbularization around $\G _{q(\g )}$ by constructing a
new smooth foliation $\cF(\G _{q(\g )} ,+)$ of $X$
that contains an enlarged Reeb component $\Omega (\G _{q(\g )},+)$ and such
that $\cF(\G _{q(\g )},+)$ coincides with the previous foliation $\cF(\G _{p(\g )},-)$
outside an embedded neighborhood of $\G _{q(\g )}$. Note that $\cF(\G _{q(\g )},+)$
also contains the enlarged Reeb component $\Omega (\G _{p(\g )},-)$ and that both
enlarged Reeb components $\Omega (\G _{p(\g )},-)$, $\Omega (\G _{q(\g )},+)$
can be assumed to be disjoint.

Finally, we repeat the above process for each transversal arc $\g _i\in \Delta $
appearing in Lemma~\ref{lemma4.1}, increasing the number of pairwise disjoint
enlarged Reeb components (two for each $\g _i$) until producing a smooth foliation $\cF'$ that contains
$2k$ enlarged Reeb components $\Omega (\G _{p(\g _i)},-)\subset
\Phi ^i_2(\overline{\D }(3\ve)\times \esf^1)$,
$\Omega (\G _{q(\g _i)},+)\subset
\Phi ^i_1(\overline{\D }(3\ve)\times \esf^1)$, $i=1,\ldots ,k$,
where $\G _1^i=\G _{q(\g )}$, $\G _2^i=\G _{p(\g )}$ for each
each open transversal arc $\g =\g _i\in \Delta $,
and the $\Phi ^i_j\colon \ov{\D }(8\ve )\times \esf^1\to V(\G ^i_j,8\ve )$
parameterize embedded tubular neighborhoods of the closed transversals $\G ^i_j$
(clearly we can assume that $\ve >0$ is common to all these tubular neighborhoods),
see Fig.~\ref{figfences}.
\begin{figure}
\begin{center}
\includegraphics[height=9cm]{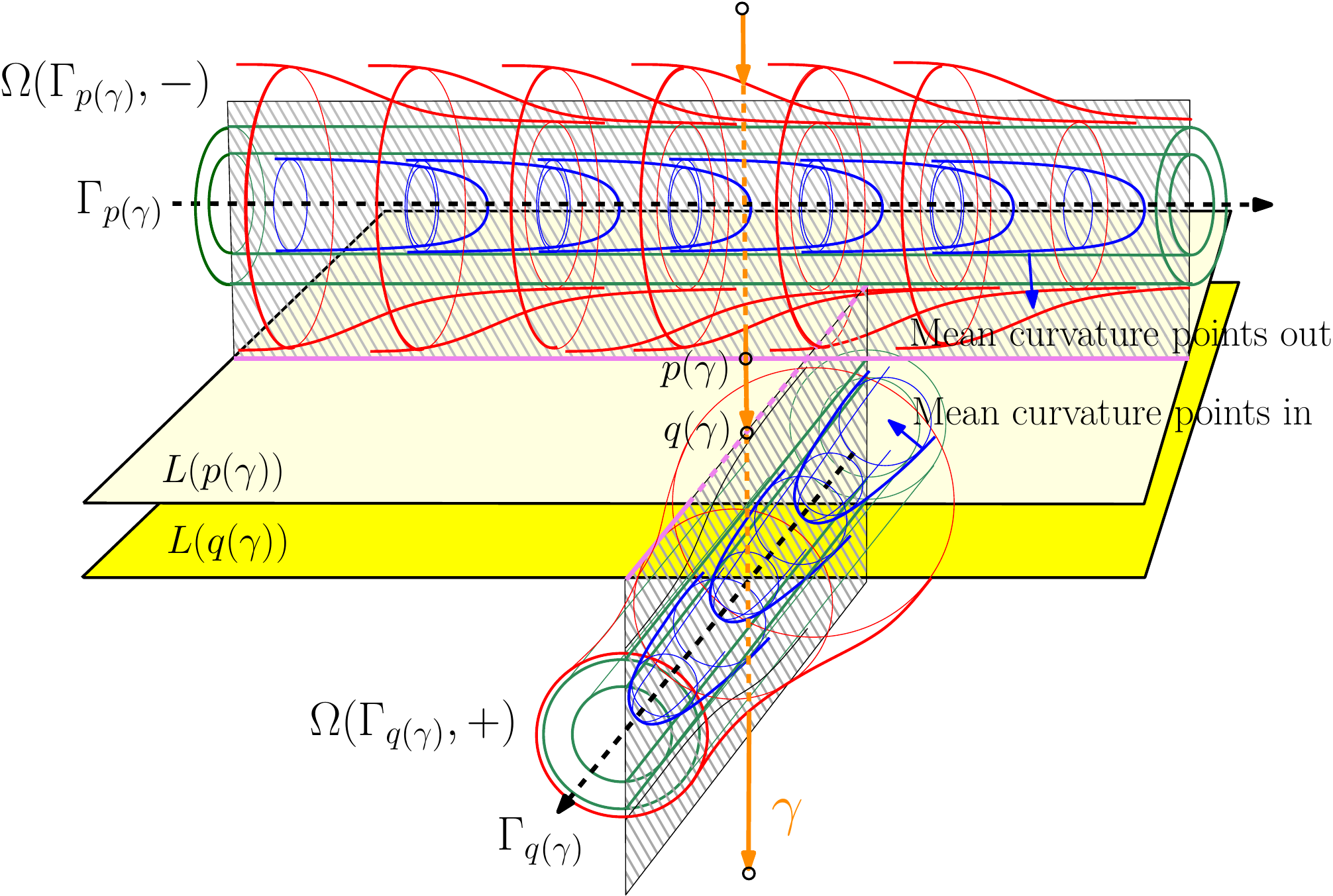}
\caption{The enlarged Reeb components $\Omega (\G_{q(\g )} ,+)$, $\Omega (\G_{p(\g )} ,-)$ crossing
each of the open transversal arcs $\g \in \Delta $.}
\label{figfences}
\end{center}
\end{figure}

\subsection{Producing the desired metric on $X$: the case $X$ is orientable.}
\label{subsecXorient}
The final step in our construction is to produce a smooth Riemannian metric on $X$ that makes
 $\cF'$ a CMC foliation. To do so we will first analyze the case that $X$ is orientable,
in which the argument is simpler and based on Oshikiri's condition (Theorem~\ref{thm2.4}).
In next section we will give a different proof based on Sullivan's arguments in~\cite{sul1}
which does not require either orientability for $X$ or Oshikiri's result.

Consider the smooth function $f$ on $X$ which equals the mean curvature function of the enlarged
Reeb foliations $\Omega (\G_{p(\g )} ,-)$, $\Omega (\G_{q(\g )} ,+)$ of $\cF'$
and is extended by zero to the complement of the union of these enlarged Reeb
components. We now check that $f$ satisfies the conditions of Theorem~\ref{thm2.4} with respect
to the foliation $\cF'$.  First observe that, with the notation introduced at the end of
Section~\ref{ConstructionF}, the portion $R(\G ^i_j,\pm )$ of each of the enlarged
Reeb components $\Omega (\G^i_j,\pm )$ given by the closure of their simply-connected
leaves (this is, after unwrapping $\G^i_j$, $R(\G ^i_j,\pm )$ lifts
to $\overline{\D}(r_1)\times \R $ with the notation of
item~(E1) just before Definition~\ref{rem3.1}), satisfies the following properties:
\begin{enumerate}[(R1)]
\item $R(\G^i_j,\pm )$ is a Reeb component of $\cF'$ (see item~3 of Definition~\ref{def3.1}).
\item $R(\G^i_j,\pm )$ is a minimal $(+)$ or $(-)$-fcd.
\item $f$ has points inside $R(\G^i_j,\pm )$ with the same sign as the character of this Reeb component.
\end{enumerate}

We now verify that there are no other minimal $(+)$ or $(-)$-fcd in $\cF'$.
Consider a smooth, foliated compact domain $D$ of $\cF'$ which is minimal under inclusion,
and which is different from the above Reeb components $R(\G^i_j,\pm)$ of $\cF'$.
%constructed in the previous paragraph.
We claim that $D$ contains a boundary
component where the transversal orientation points inward and another
 boundary component where the transversal orientation points outward (this would imply
 that there are no minimal $(+)$ or $(-)$-fcd in $\cF'$ different from the Reeb components
 $R(\G^i_j,\pm)$, which in turn implies that Oshikiri's condition holds for $f$ and $\cF'$).
 By Lemma~\ref{lemma4.1}, there exists a positively oriented transversal arc
$\g _i\in \Delta $ that intersects transversely a boundary component $\partial _1$ of $D$
 at some point and by construction, $\g _i$ contains a subarc $s$ that enters the
Reeb component $R(\G ^i_1,+)$ near the ending point of $\g_i$ and exits its related ``opposite''
Reeb component $R(\G ^i_2,-)$ near the initial point of $\g_i$.
 As $s$ also intersects $\partial _1$ at some point $s(t_0)$ and the end points of
$s$ lie outside of $D$, then $s\cap D$ must contain a further subarc $s'$ with one of
its boundary points being  $s(t_0)$ and its other end point is $s(t_1)$ which lies
in a second boundary component of $D$ along which $N_\cF$ is oppositely pointing
 with respect to $D$. Thus our claim in proved.

 Finally, when $X$ is orientable then Theorem~\ref{thm2.4} implies that there exists
 a smooth Riemannian metric $g_1$ on $X$ whose
 mean curvature function is $f$, thereby proving the
 sufficient implication of Theorem~\ref{main} in this special case for $X$.

\subsection{Producing the desired metric on $X$: the general case.}
 \label{subsec4.3}
We next prove existence of a metric $g_1$ on $X$ that makes $\cF'$ a CMC foliation,
and which coincides up to scalings with the
previously constructed smooth metrics from Section~\ref{sec3} in the finitely many
enlarged Reeb components of $\cF$
of the form $\Omega(\G_{p(\g )} ,-)$, $\Omega(\G_{q(\g )} ,+)$
with the notation at the end of Section~\ref{ConstructionF}. This proof does not
use Oshikiri's result and also works in the case $X$ is non-orientable.

The key idea is, roughly speaking, to remove part of the enlarged Reeb components
from $\cF'$ (so that a neighborhood of the resulting boundary is still a product
foliation), glue appropriately the resulting boundary leaves in pairs
to obtain a new compact $n$-manifold $\widehat{X}$
without boundary, where $\cF'$  induces a smooth codimension-one foliation $\widehat{\cF}$;
then we will construct a smooth Riemannian metric $\widehat{g}$ on
$\widehat{X}$ that makes $\widehat{\cF}$ minimal (by application of Sullivan's
Theorem possibly generalized to the non-orientable case, which will be proved in
Section~\ref{sec4.5}).
Finally, after pulling back $\widehat{g}$ to the complement of the
Reeb components and gluing this pulled back metric with the
ones constructed in Section~\ref{sec3} on the
enlarged Reeb components  (appropriately scaled by positive constants),
we will obtain the desired smooth Riemannian metric.
We now give details on this construction.

Recall that with the notation of Section~\ref{ConstructionF},
we have two closed, positively oriented
transversals $\G _1^i=\G _{q(\g )}$, $\G _2^i=\G _{p(\g )}$
associated to each open transversal arc $\g =\g _i\in \Delta $.
With the notation in Section~\ref{sec4.1}, we also have related coordinates
$\Phi ^i_j(\overline{\D }(8\ve)\times \esf^1)$
for the replaced tubular neighborhoods of the $\G_j^i$, $j=1,2$.
%Recall
%that in Section~\ref{sec4.1} we defined an enlarged Reeb component
%associated to numbers $0<r_1<r_2<1$. In order to glue this Reeb component
%with our current situation, we can exchange
%these numbers respectively by $0<2\ve <3\ve <4\ve $, in the following sense:
%\begin{enumerate}[(P1)]
%\item The ambient metric in $\overline{\D }(4\ve)\times \esf^1$ is of the form
%$(dr^2+f^2d\t ^2)+dz^2$, where $r\in [0,4\ve )$, $\t \in SO(n-2)$ denote polar
%coordinates in the disk $\overline{\D }(4\ve )$, and $z$ stands for height in
%$\R/\Z =\esf^1$.
%
%
%
%\item The foliation restricts to $\Phi ^i_j(\D (2\ve)\times \esf^1)$ as a
%Reeb-type foliation, all whose leaves are topologically $\R^{n-1}$
%and asymptotic to the boundary compact leaf $\T (2\ve )$; with respect to the
%ambient (rotationally symmetric metric), these
%leaves have the same mean curvature as $\T (2\ve )$.
%
%\item The foliation restricts to $\Phi ^i_j([\ov{\D }(4\ve )-\D (2\ve )]\times \esf^1)$
%as a product foliation by compact leaves $\T (\de )$,
%$\de \in [2\ve ,4\ve ]$. For $\de \in [2\ve ,3\ve )$, the mean curvature of
%$\T (\de )$ is varying from the mean curvature value of $\T (2\ve )$ to zero.
%The ambient metric in $\Phi ^i_j([\ov{\D }(4\ve )-\D (3\ve)]\times \esf^1)$ is a
%product metric and the $\T (\de )$ are totally geodesic
%and isometric to a fixed homothety of $\esf^{n-2}(1)\times \esf^1(1)$, for every $\de \in [3\ve ,4\ve ]$.
%\end{enumerate}

We consider the manifold
$\wh{X}$ obtained as the quotient manifold
\[
 \wh{X}=\left[X-\bigcup _{i=1}^k\left(\Phi^i_1
 (\overline{\D }(3\ve )\times \esf^1\right) \cup
\Phi^i_2 \left(\overline{\D }(3\ve )\times \esf^1\right) \right]/\sim,
\]
where $\sim$ is the equivalence relation induced by the map
$\Phi ^i_1(x,\t )\in \Phi ^i_1(\partial \D (3\ve )
\times \esf^1)\mapsto \Phi ^i_2(x,-\t)\in
\partial \Phi ^i_2(\D (3\ve )\times \esf^1)$.
It is straightforward to check:
\ben[(S1)]
\item $\wh{X}$ is orientable when $X$ is orientable.
\item The transversely oriented smooth foliation $\cF'$ on $X$ induces a transversely
oriented smooth foliation $\wh{\cF}$ of $\wh{X}$.
\item Every compact leaf of $\wh{\cF}$ is induced by a compact leaf of $\cF'$
and has a closed transversal passing through it
(namely, the quotient by the equivalence relation $\sim $
of a suitable subarc of one of the $\g _i\in \Delta $; note that for this to
make sense, we possibly need to adjust the $\t$-variable in one of the
two local `cylindrical' coordinates $(x,\t )$ defined by (\ref{eq:Phi})
so that the intersection of $\g _i$ with
$\Phi^i_1 [\overline{\D }(3\ve )\times \esf^1]$
and $\Phi^i_2 [\overline{\D }(3\ve )\times \esf^1]$
has coordinates $(x_0,0)$ in both local systems).
\item The foliation $\wh{\cF}$ is homologically taut, and hence Sullivan's
Theorem (or rather, its generalization Theorem~\ref{thm-sul} below)
implies that $\wh{X}$ admits a smooth
Riemannian metric $g_{\wh{X}}$  such that all of the leaves of the foliation
$\wh{\cF}$ are minimal.
\een

We next consider the $n$-manifold with boundary
\begin{equation}
\label{wtX}
 \wt{X}=X-\bigcup _{i=1}^{k}\left[\Phi^i_1 (\D (3\ve )\times \esf^1\right) \cup
\Phi^i_2 \left(\D (3\ve )\times \esf^1)\right] .
\end{equation}

We next describe how to glue the pulled back Riemannian metric $\wt{g}$ and the
pulled back minimal foliation $\wt{\cF}$ on $\wt{X}$, obtained respectively from $g_{\wh{X}}$,
$\wh{\cF}$ on $\wh{X}$, to the enlarged Reeb type components.

In the proof of the next assertion we will use some ideas of the proof of the main theorem
  in~\cite{sul1}, more specifically,
the paragraph in the proof of this result where Sullivan shows that
{\it homological tautness implies
geometrical tautness.} To do this, he first uses the homological tautness of
an oriented foliation $\wh{\cF}$ of a compact manifold $\wh{X}$,
to produce via the Hahn-Banach theorem a closed $(n-1)$-form
$\omega $ on $\wh{X}$ whose restriction to the leaves of $\wh{\cF}$ is positive.
Then he considers the pointwise linear map
$P_{\omega }\colon T\wh{X}\to T\wh{\cF}$ given by
\begin{equation}
\label{eq:purif}
P_{\omega }(v)\llcorner (\omega |_{\wh{F}})=(v\llcorner \omega )|_{\wh{F}}, \quad v\in T_{\wh{x}}\wh{X},
\end{equation}
where $T\wh{X}$ denotes that tangent bundle to $\wt{X}$, $T\wh{\cF}$ is the subbundle of $T\wh{X}$ tangent
to the leaves of $\wh{\cF}$, $\wh{F}$ is the hyperplane tangent
to $\wh{\cF}$ at a point $\wh{x}\in \wh{X}$ and $\llcorner $ means contraction. In particular,
$P_w(v)=v$ for all $v\in T\wh{\cF}$, hence $P_{\omega }$ is surjective.
Next he defines the {\it purification} of $\omega $ as the $(n-1)$-form on $\wh{X}$ given by
$P_{\omega }^*(\omega |_{T\wh{\cF}})$
(in our codimension-one case, $P_{\omega }^*(\omega |_{T\wh{\cF}})$ coincides with $\omega $).
Finally, Sullivan constructs the ambient smooth metric on $\wh{X}$
for which $\wh{\cF}$ is minimal by simply taking the
orthogonal direct sum
\begin{equation}
\label{eq:dec}
g_{\wh{X}}=g^1\oplus g^2
\end{equation}
of any metric $g^1$ on the $1$-dimensional
distribution $\{ \mbox{kernel}(P_\omega )\} $ with any metric $g^2$
on the subbundle $T\wh{\cF}$ whose $(n-1)$-volume form coincides with
$\omega |_{\wh{\cF}}$. In our proof of Assertion~\ref{asser:} below,
we will choose the metrics $g^1,g^2$ appearing in (\ref{eq:dec}) appropriately.

The next assertion uses the rotationally symmetric metric $ds^2$ defined by equation~(\ref{rot0})
for the function $\phi =\phi (r)$ given by (\ref{eq:rot6}) and extended to $r\in [0,1]$ as explained in
Remark~\ref{rem3.3} and in the paragraph just before this remark. Recall that by Remark~\ref{rem3.3},
given $\l >0$, the metric $ds^2+dz^2$ on $\ov{\D}(1)\times (\R /\l \Z )$ restricts to a product metric on
$[\ov{\D}(1)-\D (r_2)]\times (\R /\l \Z )$. In the sequel it will be useful
to write this product Riemannian manifold in a different manner. Consider the diffeomorphism
\[
\chi \colon [\ov{\D}(1)-\D (r_2)]\times (\R /\l \Z )\to \left( \esf^{n-2}(r_2)\times (\R /\l \Z )\right) \times [r_2,1],
\quad
\chi (r a,b)=(r_2a,b,r),
\]
where $a\in \esf^{n-2}(1)$ and $r\in [r_2,1]$. Denote by $g_\mu$ the product
metric on $\esf^{n-2}(r_2)\times (\R /\l \Z )$
that makes $\chi $ an isometry from $ds^2+dz^2$ to $g_{\mu }\times dr^2$.
Note that the function $\phi $ in (\ref{eq:rot6}) depends on the choice of $H>0$,
but this dependence does not affect the metric
$\wt{g}$ defined just after~(\ref{wtX}).
\begin{assertion}
\label{asser:}
There exist numbers $\l _1,\ldots ,\l _k>0$
such that $\wt{g}$ can be chosen so that
the each of the domains $\left( \Phi ^i_j([\ov{\D }(4\ve )-\D (3\ve)]\times \esf^1),
\wt{g}\right) $ is isometric to the Riemannian product
\[
\left( \left( \esf^{n-2}(r_2)\times (\R /\l _i\Z )\right) \times [r_2,1],
 g_{\mu }\times dr^2\right) ,
 \]
by a diffeomorphism $\psi \colon \Phi ^i_j([\ov{\D }(4\ve )-\D (3\ve)]\times \esf^1)\to
\left( \esf^{n-2}(r_2)\times (\R /\l _i\Z )\right) \times [r_2,1]$ such that the pullback by $\psi $
of the product foliation $\{ [\esf^{n-2}(r_2)\times (\R /\l _i\Z )]\times \{ \de \} \ | \ \de\in[r_2,1]\} $ is
the restriction of $\wt{\cF}$ to $\Phi ^i_j([\ov{\D }(4\ve )-\D (3\ve)]\times \esf^1)$.
\end{assertion}
\begin{proof}
We first prove the assertion under the additional hypothesis that $\cF'$ is orientable,
which, since it is transversely oriented,  would hold if $X$ were orientable. As $\cF'$
is homologically taut, then the arguments of Sullivan explained before the statement
of the assertion assure that there exists a Riemannian metric $g_{\wh{X}}$
on $\wh{X}$ that makes $\cF'$ a minimal foliation, and $g_{\wh{X}}$ has the
structure given in (\ref{eq:dec}).
Since $\wt{g}$ is the pulled back metric
of $g_{\wh{X}}$, then the decomposition (\ref{eq:dec}) can be written as well for $\wt{g}$.
Our purpose is to choose the metrics $g^1,g^2$ in (\ref{eq:dec}) (with $\wt{g}$ in the left-hand-side)
so that the assertion holds.

Recall that the linear map $P_{\omega }\colon T\wh{X}\to T\wh{F}$ defined by (\ref{eq:purif})
is surjective. Clearly $P_{\omega }$ lifts to a related surjective linear map from
$T\wt{X}$ to $T\wt{\cF}$, also denoted by $P_{\omega }$.
Consider a nowhere zero smooth vector field $V$ on
$\Phi _j^i[(\ov{\D }(4\ve )-\D (3\ve ))\times \esf^1]$ that generates the distribution $\{
\mbox{kernel}(P_{\omega })\} $. Note that $V$ is $g_{\wh{X}}$-orthogonal to the leaves of
$\wh{\cF}$ (we can consider $\Phi _j^i[(\ov{\D }(4\ve )-\D (3\ve ))\times \esf^1]$ to be
a subset of $(\wh{X},g_{\wh{X}})$ or of $(\wt{X},\wt{g})$, and in both cases we obtain
isometric compact domains with the corresponding induced metric).
After multiplying $V$ by a positive smooth function defined on
$\Phi _j^i[(\ov{\D }(4\ve )-\D (3\ve ))\times \esf^1]$, we can assume that
\begin{equation}
\label{eq:9}
\varphi _t\left[ \T ^i_j(3\ve )\right] =\T ^i_j(3\ve +t),\quad  \mbox{for any }t\in [0,\ve ],
\end{equation}
where $\{ \varphi _t\} _t$ denotes the local 1-parameter group of diffeomorphisms generated by $V$
and $\T ^i_j(3\ve +t)=\Phi _j^i(\partial \ov{\D }(3\ve +t)\times \esf^1)$, for each $t\in [0,\ve ]$, $i=1,\ldots ,k$,
$j=1,2$. Then we can identify $V=\frac{\partial }{\partial t}$. After changing the metric $g^1$,
we can also assume that $\frac{\partial }{\partial t}$ has $g^1$-length equal to $(1-r_2)/\ve $
in $\Phi _j^i[(\ov{\D }(4\ve )-\D (3\ve ))\times \esf^1]$.

Two geometrical consequences of this change of metric in the $g_{\wh{X}}$-orthogonal direction to $T\wh{\cF}$
are that the distance between $\T ^i_j(3\ve )$ and $\T ^i_j(3\ve +t)$ depends linearly on $t\in [0,\ve ]$ and
that the integral curves of $\frac{\partial }{\partial t}$ are geodesics of $g_{\wh{X}}$.
Also note that (\ref{eq:9}) implies that we can define natural ``coordinates'' $(x,t)$ in
$\Phi _j^i[(\ov{\D }(4\ve )-\D (3\ve ))\times \esf^1]$ so that $x\in \T ^i_j(3\ve )$, $t\in [0,\ve ]$,
and we can consider the projection
\[
\Pi \colon \T ^i_j(3\ve )\times [0,\ve ]\to \T ^i_j(3\ve ),\quad \Pi (x,t)=x.
\]

Once the above change of $g^1$ is done, we will change the ``tangential part'' $g^2$ of $g_{\wh{X}}$
appropriately. Consider the induced metric $g_{3\ve }$ by $g_{\wh{X}}$
on the compact hypersurface $\T ^i_j(3\ve )$.
%Let $r_2\in (0,1)$ be the positive number defined in Remark~\ref{rem3.3}.
By item (B) of Remark~\ref{rem3.3}, we can pick a positive number $\l _i$ so that
\begin{equation}
\label{Moser1}
\mbox{Vol}(\esf^{n-2}(r_2)\times (\R /\l _i\Z ),g _{\mu })=\mbox{Vol}(\T ^i_j(3\ve ),g_{3\ve }).
\end{equation}
By (\ref{Moser1}) and the first theorem in Moser~\cite{moser1}, there exists a diffeomorphism
$\xi \colon \T^i_j(3\ve )\to \esf^{n-2}(r_2)\times (\R /\l _i\Z )$
such that
\begin{equation}
\label{eq:varphi'}
\xi ^*dV_{n-1}=\omega |_{\T^i_j(3\ve )},
\end{equation}
where $dV_{n-1}$ is the $(n-1)$-volume form associated to the product metric $g_{\mu }$
defined just before Assertion~\ref{asser:}.

Now consider the smooth 1-parameter family of $(n-1)$-forms
$\{ \alpha _t\ | \ t\in [0,\ve ]\} $ on $\esf^{n-2}(r_2)\times (\R /\l _i\Z )$ defined by
\begin{equation}
\label{eq:varphi}
\left[ (\xi \circ \Pi)|_{\T ^i_j(3\ve )\times \{ t\} }\right] ^*\a _t=\omega |_{\T ^i_j(3\ve +t)}, \quad t\in [0,\ve ].
\end{equation}
Note that (\ref{eq:varphi'}) and (\ref{eq:varphi}) imply that $\a _0=dV_{n-1}$.
Since $\omega $ is a closed $(n-1)$-form, then Stokes' theorem gives that for all $t\in [0,\ve ]$,
\[
0=\int _{\Phi ^i_j[(\overline{\D}(3\ve +t)-\D (3\ve )]}d\omega =
\int _{\T ^i_j(3\ve +t)}\omega -\int _{\T ^i_j(3\ve )}\omega
\]
\[
\stackrel{(\ref{eq:varphi'}),(\ref{eq:varphi})}{=}
\int _{\esf^{n-2}(r_2)\times (\R /\l _i\Z )}\a _t-\int _{\esf^{n-2}(r_2)\times (\R /\l _i\Z )}dV_{n-1}.
\]

By Theorem 2 in~\cite{moser1}, there exists a smooth 1-parameter family of diffeomorphisms
$\phi _t\colon \esf^{n-2}(r_2)\times (\R /\l _i\Z )\to \esf^{n-2}(r_2)\times (\R /\l _i\Z )$, $t\in [0,\ve ]$,
such that $\phi _0$ is the
identity and
\begin{equation}
\label{eq:phit}
\phi _t^*dV_{n-1}=\a _t,\quad \mbox{ for each }t\in [0,\ve ].
\end{equation}
Consider the diffeomorphism
\[
\psi \colon \T ^i_j(3\ve )\times [0,\ve ]\to \left( \esf^{n-2}(r_2)\times (\R /\l _i\Z )\right) \times [r_2,1],
\qquad \psi (x,t)=\left( \phi _t(\xi (x)),{\textstyle \frac{1-r_2}{\ve }}t+r_2\right) .
\]
Calling $r=\frac{1-r_2}{\ve }t+r_2$ to the second variable in the target of $\psi $,
the vector field $\frac{\partial }{\partial r}$ on
$\left( \esf^{n-2}(r_2)\times (\R /\l _i\Z )\right) \times [r_2,1]$
is orthogonal to the product foliation
$\{ [\esf^{n-2}(r_2)\times (\R /\l _i\Z )]\times \{ \de \} \ | \ \de\in[r_2,1]\} $
with respect to the metric $g_{\mu } \times dr^2$. Therefore, the pullback
vector field $\psi ^*(\frac{\partial }{\partial r})=(\psi ^{-1})_*(\frac{\partial }{\partial r})$
is $\psi ^*(g_{\mu }\times dr^2)$-orthogonal to the pulled back foliation
\[
\{
\psi ^{-1}\left( [\esf^{n-2}(r_2)\times (\R /\l _i\Z )]\times \{ \de \} \right) \ | \ \de\in[r_2,1]\}
=\{ \T ^i_j(3\ve )\times \{ t\} \ | \ t\in [0,\ve ]\} .
\]

We claim that the $(n-1)$-volume form of the restriction of the Riemannian metric
$\psi ^*(g_{\mu }\times dr^2)$ to the leaves of $\wh{\cF}$ % $\wt{\cF}$
agrees with $\omega |_{\wh{\cF}}$; % {\wt{\cF}}$
to see this, take a $g_{3\ve +t}$-orthonormal basis $\{ e_2,
\ldots ,e_n\} $ tangent to the leaf of $\wh{\cF}$
passing through a point in $\T ^i_j(3\ve +t)$, $t\in [0,\ve ]$, positively oriented. Then,
\[
1=\left( \omega |_{\T ^i_j(3\ve +t)}\right) (e_2,\ldots ,e_n)\stackrel{(\ref{eq:varphi})}{=}
\a _t\left( (\xi \circ \Pi )_*(e_2),\ldots ,(\xi \circ \Pi )_*(e_n)\right)
\]
\[
\stackrel{(\ref{eq:phit})}{=}dV_{n-1}\left( (\phi _t\circ \xi \circ \Pi )_*(e_2),\ldots ,(\phi _t\circ \xi \circ \Pi )_*(e_n)\right) ,
\]
from where our claim
follows directly as $dV_{n-1}$ is the volume form of $(\esf^{n-2}(r_2)\times (\R /\l _i\Z ),g_{\mu })$
and the first component of $\psi $ is $\phi _t\circ \xi \circ \Pi $ (the variable $t$ is fixed as we are doing the computation
tangentially to the leaves of $\wh{\cF}$).

So, after identification by $\psi$,
Sullivan's construction can be performed with the metrics $g^1=dr^2$  and $g^2=g_{\mu }$ in
$\Phi ^i_j([\ov{\D }(4\ve )-\D (3\ve)]\times \esf^1)$ (rigorously, we should write
$g^1=\psi ^*(dr^2)$ and $g^2=\psi ^*(g_{\mu })$ on $\T ^i_j(3\ve )\times [0,\ve ]$),
and smoothly extended to the remainder of $\wt{X}$,
which finishes the proof of the assertion when $\cF'$ is orientable.

If $\cF'$ is not orientable, then the above arguments can be generalized
in a straightforward manner,
using the covering space techniques applied in the proof of
Theorem~\ref{thm-sul} below, to construct the desired
minimal metric on $\wt{X}$.
\end{proof}
\par
\vspace{.2cm}
To finish the proof of the existence of the ambient metric $g_1$
that makes $\cF'$ a CMC foliation, simply define $g_1$
to be equal to $\wt{g}$ on $\wt{X}$ and equal to the
$(SO(n-1)\times (\R/\l _i\Z))$-invariant metric
defined in Section~\ref{sec4.1} in each of the domains
$\Phi _j^i(\ov{\D }(3\ve )\times \esf^1)$, $i=1,\ldots ,k$,
$j=1,2$, where $\l _i>0$ is defined in Assertion~\ref{asser:}.
We remark that the non-zero constant values $H^i_j$
of the mean curvatures in the Reeb foliations of
$\Phi _j^i(\ov{\D }(2\ve )\times \esf^1)$ ($H^i_j<0$ for $j=1$ and $H^i_j>0$ for $j=2$)
can be chosen arbitrarily,
see Remark~\ref{rem3.3}(B).

\subsection{Sullivan's theorem for non-orientable foliations.}
\label{sec4.5}

Recall that the main theorem in Sullivan~\cite{sul1} asserts that an orientable foliation
(of any codimension) of a compact $n$-manifold is geometrically taut if and only if it is homologically taut.
We will next extend the implication `homologically taut $\Rightarrow $ geometrically taut'
in Sullivan's result by dropping the orientability assumption in the case of codimension one
(this extension was used in property (S4) of Section~\ref{subsec4.3}).
Note that no assumption on transverse orientation is made in the foliation of the next result.

Suppose that $\cF $ is an oriented, smooth, $k$-dimensional foliation of a Riemannian
$n$-dimensional manifold $(X,g)$ (not necessarily $k=n-1$),
and we denote respectively by $dV_{\cF}$, $\Pi \colon TX\to T\cF$
the volume form of the induced metric by $g$ on the leaves of $\cF $ and the $g$-orthogonal projection
of $TX$ onto $T\cF$.  In the  proof of the next theorem,
we will make use of the so called {\it Rummler's calculation.} This calculation
shows that  the smooth $k$-form $\omega $ on $X$ given by
\[
\omega _x(u_1,\ldots ,u_k)=dV_{\cF}(\Pi (u_1),\ldots ,\Pi (u_k)),\quad
x\in X, \ u_1,\ldots ,u_k\in T_xX,
\]
satisfies that it is $\cF$-closed (i.e., the restriction of $\omega $
to every $(k+1)$-dimensional submanifold of $X$ tangent to $\cF$ is closed)
if and only if the leaves of $\cF $ are minimal submanifolds (i.e., $\cF$ is
geometrically taut). In particular,
if $k=n-1$, geometrically tautness of $\cF$ is equivalent to the fact that
$\omega $ is a closed $(n-1)$-form on $X$.
For details about Rummler's calculation, see Lemma~10.5.6 in~\cite{caco1}.

\begin{theorem} \label{thm-sul}
If a smooth codimension-one foliation of a closed $n$-manifold is
homologically taut, then it is geometrically taut.
\end{theorem}
\begin{proof}
Let $\cF$ be a codimension-one, homologically taut foliation of a closed $n$-manifold $X$.
If $\cF$ can be oriented,
then Sullivan's Theorem implies that it is geometrically taut.
Suppose now that $\cF$ cannot be oriented. As orientability of $\cF $ is equivalent to
orientability of the tangent subbundle $T\cF$ of $TX$, then after passing to a
two-sheeted cover $\Pi\colon \wt{X}\to X$ of $X$, we can assume that the pulled back
foliation $\wt{\cF} $ of $\wt{X}$ by the projection $\Pi$ is orientable; this process of
passing to a two-sheeted cover does not imply that $\wt{X}$ is orientable.
Let $\sigma\colon \wt{X}\to \wt{X}$ be the order-two
covering transformation.  Also fix an orientation for the leaves of $\wt{\cF}$.

Note that since $\cF$ is homologically taut then
$\wt{\cF}$ is also homologically taut. By the proof of Sullivan's Theorem,
there exists a closed $(n-1)$-form $\omega$
on $\wt{X}$ that is positive on the oriented tangent spaces to the leaves of $\wt{\cF}$.  Note that
the pulled back $(n-1)$-form  $\sigma^*{\omega}$ is negative on
the oriented tangent spaces to the leaves of $\wt{\cF}$ and so,
the form $\wt{\omega }=\omega-\sigma^*(\omega)$ is positive on the tangent spaces
to the leaves of $\wt{\cF}$. Clearly, $\wt{\omega }$ is also closed on $\wt{X}$.
  By the proof of Sullivan's Theorem applied to $\wt{\omega }$,
there exists a metric
${g}'$ on $\wt{X}$ such that the $\wt{\omega}$ restricts to be the
$(n-1)$-volume form on the leaves of $\wt{\cF}$.
Let $\wt{g}={g}'+\sigma^*{g'}$ be the related $\sigma$-invariant metric on $\wt{X}$ and let $g$
be the corresponding quotient metric on $X$. Since $\sigma^*\wt{\omega}=-\wt{\omega}$, it follows that
the $\wt{g}$-isometry $\sigma$  leaves invariant the kernel of the projection
%purification operator
$P_{\wt{\omega}}$, defined as in~(\ref{eq:purif}).
In this setting, Rummler's calculation implies that $\wt{\cF}$ is a minimal foliation
with respect to $\wt{g}$; thus, under quotient, the foliation $\cF$ is minimal with respect to $g$,
as desired.
\end{proof}

\begin{remark}{\em
With straightforward modifications, the proof of the Theorem~\ref{thm-sul} generalizes to the case of
any smooth, $k$-dimensional foliation of a closed $n$-manifold, not just in
the case that $k=n-1$. The only difference in the proof is that after obtaining the
$k$-form $\wt{\omega}$, one replaces it by its purified
form $\wt{\omega}'=P_{\wt{\omega }}^*(\wt{\omega }|_{T\wt{\cF}})$
as described in page 208 of~\cite{sul1}, which continues to be
relatively $\cF $-closed and to satisfy $\sigma^*\wt{\omega}'=-\wt{\omega}'$.
}\end{remark}

\section{The proof of  the Structure Theorem~\ref{main2}.}
 \label{sec5}
This section is devoted to the proof of  Theorem~\ref{main2}.
 Let $(X,g)$ be a closed, connected   Riemannian $n$-manifold which
admits a non-minimal CMC foliation $\cF$.
During the proof, the reader should keep in
 mind that since $\cF$ is assumed to
 be  transversely oriented by a unit vector field $N_{\cF}$ orthogonal to the
leaves of $\cF$, then every  compact leaf $L$ of $\cF$ is two-sided
and has a closed regular neighborhood diffeomorphic to $L\times [0,1]$.
The set $\cF _0$ of compact leaves of $\cF$ is sequentially compact
in the sense that any sequence of compact leaves of $\cF$
has a subsequence that converges smoothly with multiplicity one to a compact leaf of $\cF$
(Haefliger~\cite{hae2}). In particular, the union $\cC_\cF$ of the compact leaves of
$\cF$ is a compact subset of $X$.

The divergence of $N_{\cF}$ is equal to $-(n-1)H_{\cF}$, where
we are using the notation just before the statement of Theorem~\ref{main2}. Therefore,
by  the divergence theorem, one obtains the well-known formula
 $$
 \int_X H_{\cF} \;dV=0,
 $$
which completes the proof of the first item in  Theorem~\ref{main2}.

We next prove item~\ref{it2} of the theorem. Since the foliation $\cF$ is smooth, then
$H_\cF\colon X \to \R$ is smooth as well.
By Sard's theorem and the compactness of $X$, the subset of regular values of $H_{\cF}$
is an open subset of the interval $(\min H_{\cF},\max H_{\cF})\subset \R $
whose complement has measure zero
in $[\min H_{\cF},\max H_{\cF}]$. By the implicit
 function theorem and the compactness of
$X$, for each regular value $H$ of $H_\cF$, $H_{\cF}^{-1}(H)$ consists of
a finite number of compact leaves of $\cF$ contained in $\Int(\cC_\cF)$.
This completes the proof of item~\ref{it2}.

We now proceed with the proof of item~\ref{it3}.
Since manifolds are second countable and $\cC_\cF$ is a compact subset of $X$,
then $X-\cC_\cF$ has a countable number of
components, all of which are open.
It follows that if
$H_\cF$ restricted to a component $\Delta$ of  $X-\cC_\cF$ were not constant, then by Sard's
Theorem there would be a regular value $H_0$ of $H_\cF$  different from any
of the finite number of values of  $H_\cF$ on the finite
number of compact boundary components of $\Delta$.
By item~\ref{it2} of this theorem, $\Delta \cap (H_\cF)^{-1}(H_0)$
contains a compact leaf of $\cF$;
this is a contradiction since $\Delta$ lies in $X-\cC_\cF$.
Thus, $H_{\cF}$ is constant in every component of
$X-\cC_\cF$, from which the first part of the first sentence in item~\ref{it3}
of the theorem follows by the continuity
of $H_{\cF}$.
In particular, except for the countable subset of $H_\cF(X)$ corresponding to the
set of values of
$H_\cF$ on the components of $X-\cC_\cF$, every leaf of $\cF$ with mean
curvature different from one of these special values is compact.  Furthermore,
if $H_\Delta \in \R$ is the value of
$H_\cF$ on a component $\Delta$ of $X-\cC_\cF$, then
by the continuity of $H_\cF$, $H_\Delta$
is the value of the mean curvature of any compact leaf in the
boundary of  $\Delta$.
Thus, for every $H\in H_\cF(X)$ there exists
at least one compact leaf of $\cF$ of  mean curvature $H$.

To complete the proof of item~\ref{it3}, it remains to demonstrate
that every leaf in the closure of $X-\cC_\cF$ is stable.
To do this, first consider a component $\Delta$ of  $X-\cC_\cF$.
We have already proved that the leaves
in the closure of $\Delta$ have the same constant mean curvature $H$. Since the closure of
 $\Delta$ in $X$ has the structure of an $H$-lamination of $X$ where every leaf
 is a limit leaf, then the main theorem in~\cite{mpr18}
 implies that every leaf in the closure of $\Delta$ is a stable $H$-hypersurface.
Suppose that $L$ is a leaf in the closure
of $X-\cC_\cF$ which is not a leaf in the closure of any component of $X-\cC_\cF$;
in this case, every point $x\in L$ is the limit in $X$ of a sequence of
points $x_n\in X$, each of  which lies in the boundary
$\partial \Delta _n$ of a component $\Delta _n$ of $X-\cC_\cF$.
In this case,
Haefliger's compactness result for the set of closed leaves of $\cF$ implies that
$L$ is compact and it is the smooth limit (with multiplicity one)
of a sequence of compact $H_n$-stable leaves $L_n\subset \partial \Delta _n$.
By the continuity of $H_{\cF}$, the $H_n$ converge to the mean curvature of $L$.
Since for $n$ large any smooth unstable subdomain in $L$
can be lifted normally to a smooth
 unstable subdomain in $L_n$, the instability of $L$ would contradict
the assumption that the $L_n$ are  $H_n$-stable leaves.
 Therefore, $L$ must also be stable, which completes the proof of item~\ref{it3} of the theorem.

To prove item~\ref{it4}, suppose that $L$ is a leaf of $\cF$. The local product
structure of a foliation clearly implies that the set $A=\{ p\in L \ | \ (\nabla H_{\cF})(p)=0\} $
is open in $L$. Since $A$ is clearly closed in $L$ and $L$ is connected, then $A=L$ or
$A=\mbox{\O }$. In particular, if $L$ contains a regular point of $H_{\cF}$ then $A=L$
and so, $L$ consists entirely of regular points of $H_{\cF}$.
By item~\ref{it3}, the closure of  $X-\cC_\cF$ consists entirely of
critical points of $H_\cF$.  Since every non-compact leaf of $\cF$
is contained in $X-\cC_\cF$, then we conclude that $L$ must
be compact. As the set of regular points
of $H_{\cF}$ is open, then the same arguments show that $L$
lies in the interior of ${\cC}_{\cF}$ and that
$\nabla H_\cF$ is non-zero in a neighborhood $U(L)$ of $L$ in $X$.
By an application of the implicit function theorem
to $H_\cF$, $U(L)$ can be taken so that $\cF $ restricts to
$U(L)$ as a product foliation of compact leaves diffeomorphic to $L$. This product foliation
can be considered to be a smooth normal variation $L_t$ of $L$
through compact leaves of $\cF$, whose variational field is $V=fN_{\cF}$
for $f=g(\left. \frac{d}{dt}\right| _{t=0}L_t,N_{\cF})$ and with
\[
Jf=(n-1)\left. \frac{d}{dt}\right| _{t=0}H(L_t)\neq 0
\]
everywhere on $L$, where $J=\Delta +\| \sigma \| ^2+\mbox{Ric}(N_{\cF})$ is the Jacobi
operator of $L$ (here $\| \sigma \| ^2$ denotes the square of the norm of the second fundamental
form of $L$ and Ric the ambient Ricci curvature).
The foliation property of the variation $t\mapsto L_t$
insures that $f$ has constant sign on $L$, say $f>0$. If $Jf<0$ on $L$, then Lemma~2.1
in~\cite{mpr19} insures that $L$ is stable (in fact, with nullity zero). Conversely, if
$Jf>0$ on $L$, then the index form $Q(f,f)=-\int _LfJf$ is strictly negative, and so
$L$ is unstable. Thus, item~4a of Theorem~\ref{main2} holds.

Suppose now that $L$ is a leaf of $\cF$ that is disjoint from the regular points of $H_\cF$.
One possibility for $L$ is that it is contained in the closure of $X-\cC_\cF$, in which case
 item~\ref{it3} implies $L$  is stable, and so, the
index of $L$ is zero.   Otherwise $L$ lies in the interior of $\cC_\cF $,
which means that $L$ is a limit leaf of
the CMC lamination of $X$ consisting of compact leaves of $\cF$.
%Then by the arguments at the end of the proof
%of item~\ref{it3} of this theorem,
As $L$ lies in a 1-parameter family $\{L_t\}$
of compact leaves  whose mean curvatures $H_t$ at $L$ agree up
to first order with the mean curvature of $L$ (since $L$ is a subset of
critical points of $H_\cF$),  the normal variational vector field $F$ to the variation
at the leaf $L$
is a nowhere zero Jacobi field  on $L$. It follows in this case that
$L$ is stable and has nullity one.  This completes the proof of item~\ref{it4}b.

Finally we prove item~\ref{it5}.
Suppose that $L$ is a leaf of $\cF$ with mean curvature in
$\{ \min H_\cF,  \max H_\cF \}$.  Then $L$ is stable
by item~\ref{it4}b. Next consider a
sequence of regular values $r_n$ of $H_\cF$ that
are tending  to, but smaller than, the value $\max H_\cF$.
By item~\ref{it2}, there exist compact leaves $L_n$ of $\cF$
with $H_\cF(L_n)=r_n$, and by compactness of $\cC_\cF$,
after replacing by a subsequence, the $L_n$ converge to a compact
leaf $L$ with mean curvature  $\max H_\cF$.  Since $L$ is
disjoint from the regular values of $H_\cF$ by item~4a, then the last part of item~\ref{it4}b
implies that the nullity of $L$ is one. A similar argument
proves the existence of a compact stable leaf $L'$
of nullity one and mean curvature $\min H_\cF$, which
completes the proof of Theorem~\ref{main2}. \vspace{.3cm}

\center{William H. Meeks, III at profmeeks@gmail.com\\
Mathematics Department, University of Massachusetts, Amherst, MA 01003}
\center{Joaqu\'\i n P\'{e}rez at jperez@ugr.es\\ % \qquad\qquad Antonio Ros at aros@ugr.es\\
Department of Geometry and Topology, University of Granada, Granada, Spain}

\bibliographystyle{plain}
\bibliography{bill}
\end{document}